\documentclass[intlimits,reqno]{amsart}

\usepackage{amsmath, amsthm, amssymb, amsfonts}
\usepackage{graphicx}
\usepackage[english]{babel}
\usepackage[usenames]{color}
\usepackage{ifpdf}
\usepackage{dsfont}
\usepackage{enumerate}
\usepackage{amsbsy}

\ifpdf
\usepackage{hyperref}
\else
\usepackage[hypertex]{hyperref}
\hypersetup{backref, colorlinks=true, bookmarks}
\fi

\newtheorem{theorem}{Theorem}[section]
\newtheorem{proposition}{Proposition}[section]
\newtheorem{lemma}[theorem]{Lemma}

\newtheorem{definition}[theorem]{Definition}

\newtheorem{corollary}[theorem]{Corollary}
\theoremstyle{definition}

\numberwithin{equation}{section}
\newcommand{\R}{\mathbb{R}}
\newcommand{\C}{\mathbb{C}}

\newcommand{\Z}{\mathbb{Z}}
\newcommand{\N}{\mathbb{N}}
\newcommand{\eps}{\epsilon}

\newcommand{\norm}[1]{\| #1 \|}
\newcommand{\norminf}[1]{\| #1 \|_\infty}    

\newcommand{\set}[1]{\left\{#1\right\}}
\newcommand{\E}[1]{\mathbb{E}\left[#1\right]}
\newcommand{\Pb}[1]{\mathbb{P}\left[#1\right]}

\newcommand{\Et}[1]{\mathbb{E}_\theta\left[#1\right]}
\newcommand{\Pt}[1]{\mathbb{P}_\theta\left[#1\right]}

\newcommand{\Var}[1]{\ensuremath{\var\left(#1\right)}}
\newcommand{\intpart}[1]{\ensuremath{\left[ #1 \right]}}
\newcommand{\fracpart}[1]{\ensuremath{\left\{ #1 \right\}}}

\newcommand{\vB}{\boldsymbol{\varphi}}
\renewcommand{\Re}{\mathrm{Re}}
\renewcommand{\Im}{\mathrm{Im}}

\DeclareMathOperator{\cov}{Cov}
\DeclareMathOperator{\var}{Var}

\begin{document}

\title[The Characteristic Polynomial at Different Points]{The Characteristic Polynomial of a Random Permutation Matrix at Different Points} 

\author[K. Dang]{Kim Dang}
\address{K. Dang\\
Department of Mathematics \\
Yale University\\ 
10 Hillhouse Avenue\\ 
New Haven\\ CT 06511-6810, United States}
\email{kim.dang@yale.edu}

\author[D.~Zeindler]{Dirk Zeindler$^1$}
\address{D. Zeindler, Department of Mathematics, University of Bielefeld, Bielefeld D-33501, Germany}
\email{zeindler@math.uni-bielefeld.de}
\thanks{$^1$support by the SNF (Swiss National Science Foundation)}

%
%
%
%

%


\date{\today}

\maketitle

\begin{abstract}
We consider the logarithm of the characteristic polynomial of random permutation matrices, evaluated on a finite set of different points. The permutations are chosen with respect to the Ewens distribution on the symmetric group. We show that the behavior at different points is independent in the limit and are asymptotically normal. Our methods enables us to study also the wreath product of permutation matrices and diagonal matrices with iid entries 
and more general class functions on the symmetric group with a multiplicative structure..
\end{abstract}

\tableofcontents
%
\newpage
\section{Introduction}

The characteristic polynomial of a random matrix is a well studied object in Random Matrix Theory (RMT) (see for example \cite{ashkan}, \cite{PhysRevLett.75.69}, \cite{HKO01}, \cite{HKOS}, \cite{snaith}, \cite{7459}, \cite{EJP2010-34}, \cite{DZ-CLT}).
An important result due Keating and Snaith \cite{snaith} on $n\times n$ CUE matrices is that
the imaginary and the real part of the logarithm of the characteristic polynomial converge jointly
in law to independent standard normal distributed random variables, after normalizing by
$\sqrt{(1/2)\log n}$. Hughes, Keating and O'Connell refined this result in \cite{HKO01}: evaluating
the logarithm of the characteristic polynomial, normalized by $\sqrt{(1/2)\log n}$, for a discrete
set of points on the unit circle, this leads to a collection of i.i.d. standard (complex) normal
random variables.\\

In \cite{HKOS}, Hambly, Keevash, O'Connell and Stark give a Gaussian limit for the logarithm of the characteristic polynomial of random permutation matrices under uniform measure on the symmetric group.
This result has been extended by Zeindler in \cite{DZ-CLT} to the Ewens distribution on the symmetric group and to the logarithm of multiplicative class functions, introduced in \cite{associated_class}.\\

In this paper, we will generalize the results in \cite{HKOS} and \cite{DZ-CLT} in two ways. First, we follow the spirit of \cite{HKO01} by considering the behavior of the logarithm of the characteristic polynomial of a random permutation matrix at different points $x_1,\dots, x_d$. Second, we state CLT's for the logarithm of characteristic polynomials for matrix groups related to permutation matrices, such as some Weyl groups \cite[section 7]{associated_class} and of the wreath product $\mathbb{T}\wr S_n$ \cite{wieandPaper2}, where $\mathbb{T} = \{z\in\C; |z| =1\}$.

In particular, we consider $n\times n$-matrices $M=(M_{ij})_{1\leq i,j\leq n}$ of the following form: for a permutation $\sigma\in S_n$ and  a complex valued random variable $z$,%
\begin{equation}
\label{eq_def_wreath_product_matrix}
 M_{ij}(\sigma,z):=z_i\delta_{i,\sigma(j)},
\end{equation}
where $z_i$ is a family of i.i.d. random variables such that $z_i \stackrel{d}{=} z$, $z_i$ independent of $\sigma$. Here, $\sigma$ is chosen with respect to the Ewens distribution, i.e.
\begin{align}
\label{eq_Ewens_distribution}
\Pt{\sigma}
:=
\frac{\theta^{l_\sigma}}{\theta(\theta+1)\dots(\theta+n-1)},
\end{align}
for fixed parameter $\theta>0$ and $l_\sigma$ being the total number of cycles of $\sigma$. The \emph{Ewens measure} or \emph{Ewens distribution} is a well-known measure on the the symmetric group $S_n$, appearing for example in population genetics \cite{MR0325177}. It can be viewed as a generalization of the uniform distribution (i.e $\Pb{A}=\frac{|A|}{n!}$) and has an additional weight depending on the total number of cycles. The case $\theta=1$ corresponds to the uniform measure. Matrices $M(\sigma,z)$ of the form \eqref{eq_def_wreath_product_matrix} can be viewed as generalized permutation matrices $M(\sigma)=M(\sigma,1)$, where the $1$-entries are replaced by i.i.d. random variables.
Also, it is easy to see that elements of the wreath product $\mathbb{T} \wr S_n$ with $\mathbb{T} = \{z\in\C; |z| =1\}$ (see \cite{wieandPaper2}  and \cite[section~4.2]{associated_class}) or elements of some Weyl groups (treated in \cite[section 7]{associated_class}) are of the form \eqref{eq_def_wreath_product_matrix}. In this paper, we will not give any more details about wreath products and Weyl groups, since we do not use group structures. \\\\
We define the function $Z_{n,z}(x)$ by
\begin{align}
\label{eq_def_Znx_simple_1}
Z_{n,z}(x)
:=
\det\bigl(I-x^{-1}M(\sigma,z)\bigr), \quad x\in\C^*.
\end{align}
Then, the characteristic polynomial of $M(\sigma,z)$ has the same zeros as $Z_{n,z}(x)$. We will
study the characteristic polynomial by identifying it with $Z_{n,z}(x)$, following the convention of
\cite{associated_class}, \cite{EJP2010-34} or \cite{DZ-CLT}.\\
By using that the random variables $z_i$, $1\leq i\leq n$ are i.i.d., a simple computation shows the following equality in law (see \cite{associated_class}, Lemma 4.2):
\begin{equation}
\label{eq_Zn_explicit_2}
Z_{n,z}(x)
\stackrel{d}{=}
\prod_{m=1}^n\prod_{k=1}^{C_m}(1-x^{-m} T_{m,k}),
\end{equation}
where $C_m$ denotes the number of cycles of length $m$ in $\sigma$ and $(T_{m,k})_{m, k \geq 1}$ is a family of independent random variables, independent of $\sigma \in S_n$, such that
\begin{equation}
T_{m,k} \stackrel{d}{=}\prod_{j=1}^{m} z_j. 
\end{equation}

Note that the characteristic polynomial  $Z_{n,z}(x)$ of $M(\sigma,z)$ depends strongly on the
random variables $C_m$ ($1\leq m\leq n$). The distribution of $(C_1,C_2,\cdots,C_n)$ with respect to
the Ewens distribution with parameter $\theta$ was first derived by Ewens (1972), \cite{MR0325177}.
It can be computed, using the inclusion-exclusion formula, \cite[chapter 4, $(4.7)$]{barbour}. \\

We are interested in the asymptotic behavior of the logarithm of \eqref{eq_def_Znx_simple_1} and
therefore, we will study the characteristic polynomial of $M(\sigma, z)$ in terms of \eqref{eq_Zn_explicit_2}, by choosing the branch of logarithm in a suitable way. In view of \eqref{eq_Zn_explicit_2}, it is natural to choose it as follows: 

\begin{definition}
\label{def_log_Zn,Wn_1dim_1}
Let $x = e^{2 \pi i \varphi} \in \mathbb{T}$ be a fixed number and $z$ a $\mathbb{T}$--valued random variable. Furthermore, let $(z_{m,k})_{m,k=1}^\infty$ and $(T_{m,k})_{m,k=1}^\infty$ be two sequences of independent random variables, independent of $\sigma \in S_n$ with
\begin{align}
  z_{m,k} \stackrel{d}{=} z \ \text{ and } \  T_{m,k} \stackrel{d}{=} \prod_{j=1}^m z_{j,k}.
\end{align}
We then set
\begin{align}
 \label{eq_log_Zn_1}
  \log\bigl(Z_{n,z}(x) \bigr)
  &:=
  \sum_{m=1}^n\sum_{k=1}^{C_m} \log (1-x^{-m}T_{m,k}),
\end{align}
where we use for $\log(.)$ the principal branch of logarithm. We will deal with negative values as follows: $\log(-y)=\log y+i\pi$, $y\in\R_+$. 
\end{definition}
Note, that it is not necessary to specify the logarithm at $0$ since our assumptions in the cases studied always ensure that
this occurs only with probability 0 (see Theorem~\ref{thm_CLT_W1_1dim_general} and Theorem~\ref{thm_CLT_W2_1dim_general}).

In this paper, we show that under various conditions, $\log Z_{n,z}(x)$ converges to a complex standard Gaussian distributed random variable after normalization and the behavior at different points is independent in the limit. Moreover, the normalization by $\sqrt{(\pi^2/12)\theta\log n}$ is independent of the random variable $z$. This covers the result in \cite{HKOS} for $\theta=1$ and $z$ being deterministic equal to $1$. We state this more precisely:

\begin{proposition}
\label{prop_Zn_general_1dim}
Let $S_n$ be endowed with the Ewens distribution with parameter $\theta$, $z$ a $\mathbb{T}$-valued random variable and $x \in \mathbb{T}$ be not a root of unity, i.e. $x^m \neq 1 $ for all $m\in\Z$.

Suppose that $z$ is uniformly distributed. Then, as $n\rightarrow\infty$,
 \begin{align}
  \label{eq CLT_z_uniform_Zn_real}
    &\frac{\Re\left(\log\bigl(Z_{n,z}(x) \bigr)\right)}{\sqrt{\frac{\pi^2}{12}\theta\log n}}
    \stackrel{d}{\longrightarrow}
     N_R \quad\text{and}\\
      \label{eq CLT_z_uniform_Zn_im}
     &\frac{\Im\left(\log\bigl(Z_{n,z}(x) \bigr)\right)}{\sqrt{\frac{\pi^2}{12}\theta\log n}}
    \stackrel{d}{\longrightarrow}
     N_I,
  \end{align}
  with $N_R,N_I \sim \mathcal{N} \left(0,1 \right)$.
\end{proposition}

In Proposition~\ref{prop_Zn_general_1dim} $\Re\left(\log\bigl(Z_{n,z}(x) \bigr)\right)$ and $\Im\left(\log\bigl(Z_{n,z}(x) \bigr)\right)$ are converging to normal random variables without centering. This is due to that the expectation is $o(\sqrt{\log n})$. This will become more clear in the proof (see Section~\ref{sec_CLT_in_1dim}).\\
Furthermore, we state a CLT for $\log Z_{n,z}(x)$, evaluated on a finite set of different points $\{x_1,\dots,x_d\}$.

\begin{proposition}
\label{prop_CLT_Zn_general_ddim}
  Let $S_n$ be endowed with the Ewens distribution with parameter $\theta$, $\overline{z} = (z_1, \dots, z_d)$ be a $\mathbb{T}^d$-valued random variable 
and $x_1=e^{2\pi i \varphi_1} ,\dots, x_d = e^{2\pi i \varphi_d} \in \mathbb{T}$ be such that $1, \varphi_1,\dots,\varphi_d$ are linearly independent over $\Z$.

Suppose that $z_1, \dots, z_d$ are uniformly distributed and independent. Then we have, as $n\to\infty$,
\begin{align*}
\frac{1}{\sqrt{\frac{\pi^2}{12}\theta\log n}}
  \left(
\begin{array}{c}
 \log(Z_{n,z_1}(x_1)\bigr)\\
\vdots\\
\log \bigl(Z_{n,z_d}(x_d) \bigr)
\end{array}
\right)
\stackrel{d}{\longrightarrow}
  \left(
\begin{array}{c}
N_1\\
\vdots\\
N_d
\end{array}
\right)
\end{align*}
with $\Re(N_1),\dots,\Re(N_d), \Im(N_1),\dots,\Im(N_d)$ independent standard normal distributed random variables.
\end{proposition}

Note that $z_1, \dots, z_d$ are not equal to the family $(z_i)_{1\leq i\leq n}$ of i.i.d. random
variables in \eqref{eq_def_wreath_product_matrix}. In fact, we deal here with $d$ different families
of i.i.d. random variables, where the distributions are given by $z_1, \dots, z_d$ and we thus deal also with 
$d$ different matrices, all basing on the same permutation matrix. We will treat
this more carefully in Section~\ref{sec_CLT_in_ddim}.
A remaining open question is the joint behaviour at different points of $\log\bigl(Z_{n,z}(x)\bigr)$ with $z$ uniform, but
we expect also in this case a central limit theorem.

Proposition~\ref{prop_CLT_Zn_general_ddim} shows that the characteristic polynomial of the random matrices $M(\sigma, z)$ follows
the tradition of matrices in the CUE, if evaluated at different points, due to the result by
\cite{HKO01}. 
Moreover, the proof of Proposition~\ref{prop_CLT_Zn_general_ddim} can also be used for regular
random permutation matrices, i.e. $M(\sigma, 1)$, but requires further assumptions on the points
$x_1, \dots, x_d$. We state this more precisely:

\begin{proposition}
\label{prop_CLT_Zn_trivial_ddim}
 Let $S_n$ be endowed with the Ewens distribution with parameter $\theta$ 
 and $x_1=e^{2\pi i \varphi_1} ,\dots, x_d = e^{2\pi i \varphi_d} \in \mathbb{T}$ be 
 pairwise of finite type (see Definition~\ref{def_finite_type_pairwise}).

We then have for $z_1=\cdots=z_d=1$, as $n\rightarrow\infty$,
\begin{align*}
\frac{1}{\sqrt{\frac{\pi^2}{12}\theta\log n}}
  \left(
\begin{array}{c}
 \log(Z_{n,1}(x_1)\bigr)\\
\vdots\\
\log \bigl(Z_{n,1}(x_d) \bigr)
\end{array}
\right)
\stackrel{d}{\longrightarrow}
  \left(
\begin{array}{c}
N_1\\
\vdots\\
N_d
\end{array}
\right)
\end{align*}
with $\Re(N_1),\dots,\Re(N_d), \Im(N_1),\dots,\Im(N_d)$ independent standard normal distributed random variables.
\end{proposition}

In fact, our methods allow us to prove much more. First, we are able to relax the conditions in the
Propositions~\ref{prop_Zn_general_1dim}, \ref{prop_CLT_Zn_general_ddim}
and~\ref{prop_CLT_Zn_trivial_ddim} above. Also, these results on $\log Z_{n,z}(x)$ follow as
corollaries of much more general statements (see Section~\ref{sec_examples}). Indeed, the methods
allow us to prove CLT's for multiplicative class functions. Multiplicative class functions have been
studied by Dehaye and Dehaye-Zeindler, \cite{associated_class}, \cite{DZ-CLT}.\\

Following \cite{associated_class}, we present here two different types of multiplicative class functions. The first multiplicative class function is defined as follows.
\begin{definition}
\label{def_W1_poly}
Let $z$ be a complex valued random variable and $f:\C \to \C$ be given.
We then define the \emph{first multiplicative class function associated to}~$f$
as the random variable $W^1(f)(x)$ on $S_n$ with
\begin{align}
\label{eq_def_W1}
W^1(f)(x)
&=
W_{z}^{1,n} (f)(x)(\sigma)
:=
\prod_{m=1}^{n} \prod_{k=1}^{C_m}
f\left(z_m x^m \right),
\end{align}
where $\sigma\in S_n$, $z_m \stackrel{d}{=} z$, $z_m$ i.i.d. and independent of $\sigma$.
%
\end{definition}
The second multiplicative class function is directly motivated by the expression
\eqref{eq_Zn_explicit_2} and is a slightly modified form of \eqref{eq_def_W1}.

\begin{definition}
\label{def_W2}
Let $z$ be a complex valued random variable and $f:\C \to \C$ be given.
We then define the \emph{second multiplicative class function associated to}~$f$
as the random variable $W^2(f)(x)$ on $S_n$ with
%
\begin{align}
\label{eq_def_W2}
W^2(f)(x)
&=
W_{z}^{2,n} (f)(x)(\sigma)
:=
\prod_{m=1}^{n} \prod_{k=1}^{C_m}
f\left( x^m T_{m,k} \right),
\end{align}
where $\sigma\in S_n$,  $T_{m,k}$ is a family of independent random variables, $T_{m,k}\stackrel{d}{=}\prod_{j=1}^{m} z_j$ and $z_j\stackrel{d}{=}z$, for any $1\leq j\leq n$.
\end{definition}
It is obvious from \eqref{eq_Zn_explicit_2} and \eqref{eq_def_W2} that $Z_{n,z}(x)$ is the special case $f(x)=1-x^{-1}$ of $W^2(f)(x)$. This explains, why results on the second multiplicative class function cover in general results on $\log Z_{n,z}(x)$.

We postpone the statements of the more general theorems on multiplicative class functions to Section~\ref{sec_examples}.\\

At this point it is natural to ask if there are any other important examples of multiplicative class function than $f(x) = 1 -x^{-1}$.
For instance, consider the matrices $\mathcal{S}=(S_{ij})$ with
\begin{equation}
\label{eq_def_wreath_product_matrix_sym}
 S_{ij}(\sigma):=\delta_{i,\sigma(j)} +  \delta_{i,\sigma^{-1}(j)}.
\end{equation}
The matrix $\frac{1}{2}\mathcal{S}$ is the symmetric part of $M(\sigma,1)$ and can be interpreted as the adjacency matrix of a 2-regular graph.
It follows from \eqref{eq_def_wreath_product_matrix} that $\mathcal{S} = M(\sigma,1) +  \overline{M(\sigma,1)}^T$.
Since $M(\sigma,1)$ is a unitary matrix with reel entries, we see that $M(\sigma,1)^{-1}  = \overline{M(\sigma,1)}^T$ and thus $M(\sigma,1)$ and $\overline{M(\sigma,1)}^T$ commute.
Therefore $\frac{1}{2}\mathcal{S}$ has the same eigenbasis as $M(\sigma,1)$ but the eigenvalues are projected to the real axis.
If $\sigma$ is a cycle of length $n$, then the eigenvalues of the corresponding permutation matrix are $\exp(2\pi i m/n)$ with $0\leq m < n$ (see \cite{benarousdang} or \cite{EJP2010-34}).
Thus the eigenvalues of $\mathcal{S}(\sigma)$ are
\begin{align}
\label{eq:zeros_sym} 
2\cos(0), \, 2\cos\left(\frac{2\pi i}{n}\right), \, \dots,\, 2\cos\left(\frac{2\pi i(n-1)}{n}\right).
\end{align}
Using the Chebyshev polynomials $T_n(x)$ with the trigonometric definition \\$T_n(\cos(y))=\cos(ny)$,
one immediately sees that the zeros of $T_n(x/2)-1$ are given by \eqref{eq:zeros_sym}.
Furthermore, $T_n(x/2)-1$ is polynomial of degree $n$ with leading coefficient $1/2$.
We write $x=2\cos(\alpha)$ for $x\in[-2,2]$ and $y=e^{i\alpha}$. This gives
\begin{align}
\det(S_{\sigma}-xI) 
&= 
\prod_{m=1}^n \bigl(2(1-T_m(x/2))\bigr)^{C_m} 
=
\prod_{m=1}^n \bigl(2(1-\cos(m\alpha))\bigr)^{C_m} \\
&=
\prod_{m=1}^n \left(2-e^{im\alpha}+e^{-im\alpha}\right)^{C_m}\nonumber\\
&=
\prod_{m=1}^n \left(2-y^{m}+y^{-m}\right)^{C_m}
= W^{2,n}_1(f_S)(y)
\nonumber\nonumber
\end{align}
with $f_S(y) = 2-y^{m}+y^{-m}$. Similarly for the anti-symmetric part $\mathcal{A}$ of $M(\sigma,1)$
\begin{align}
\det(2\mathcal{A}-xI) = W^{2,n}_1(f_A)(y), \qquad \text{with }f_A(y) = 2-y^{m}-y^{-m}.
\end{align}

%
%
%

For the proofs we will make use of similar tools as in \cite{HKOS} and \cite{DZ-CLT}. These tools include the Feller Coupling, uniformly distributed sequences and Diophantine approximations. 

The structure of this paper is as follows: In Section~\ref{sec_preliminaries}, we will give some background of the Feller Coupling. Moreover, we recall some basic facts on uniformly distributed sequences and Diophantine approximations. In Section~\ref{sec_CLT_on_Sn}, we state some auxiliary CLT's on the symmetric group, which we will use in Section~\ref{sec_examples} to prove our main results for the characteristic polynomials and more generally, for multiplicative class functions.

\section{Preliminaries}
\label{sec_preliminaries}

\subsection{The Feller coupling}
\label{sec_Feller_coupling}

The reason why we expand the characteristic polynomial of $M(\sigma,z)$ in terms of the cycle counts of $\sigma$
 as given in \eqref{eq_Zn_explicit_2} is the fact that the asymptotic behavior of the numbers of cycles with 
length $m$ in $\sigma$, denoted by $(C_m)_{1\leq m\leq n}$, has been well-studied, for example by \cite{barbour} or \cite{MR0325177}.
 In particular, the random variables $C_m$ converge as $n\to\infty$ to independent Poisson random variables $Y_m$ with mean $\theta/m$, $m\geq 1$.
We use in this paper the Feller coupling, which is an important probabilistic tool and allows to define all random variables $C_m$ and $Y_m$ on the same space.
We give here only a very brief overview. Details can be found for instance in \cite{barbour}, \cite{MR0325177}, \cite{watterson}.

%
%
%
%

\begin{definition}
\label{def_cm_feller}
Let $\xi_i$ be independent Bernoulli random variables for $i\geq 1$ with
$$\Pb{\xi_i=1} =\frac \theta{\theta+i-1} \quad \textrm{and} \quad \Pb{\xi_i=0} = \frac{i-1}{\theta+i-1}.$$
Define $C_m^{(n)}(\xi)$ to be the number of m-spacings in $1\xi_2\cdots \xi_n 1$ and
$Y_m(\xi)$ to be the number of m-spacings in the limit sequence, i.e.
\begin{equation}
C_m^{(n)}(\xi)=\sum_{i=1}^{n-m}\xi_i(1-\xi_{i+1})\dots(1-\xi_{i+m-1})\xi_{i+m}+\xi_{n-m+1}(1-\xi_{n-m+2})\dots(1-\xi_n)
\end{equation}

and
\begin{equation}
Y_{m}(\xi)=\sum_{i=1}^\infty \xi_i(1-\xi_{i+1})\dots(1-\xi_{i+m-1})\xi_{i+m}.
\end{equation}
\end{definition}
Then the following theorem holds (see \cite[Chapter 4, p. 87]{barbour} and \cite[Theorem 2]{MR1177897}).
\begin{theorem}
\label{thm_feller_conv}
Under the Ewens distribution, we have that
          \begin{itemize}
            \item The above-constructed $C_m^{(n)}(\xi)$ has the same distribution as the variable $C_m^{(n)}=C_m$, the number of cycles of length $m$ in $\sigma$.
            \item $Y_m(\xi)$ is a.s. finite and Poisson distributed with $\E{Y_m(\xi)}=\frac{\theta}{m}$.
            \item All $Y_m(\xi)$ are independent.
             \item For any fixed $b\in\N$,
            $$\Pb{(C_1^{(n)}(\xi),\cdots, C_b^{(n)}(\xi))\neq(Y_1(\xi),\cdots ,Y_b(\xi))}\to 0\ (n\to\infty).$$
             \end{itemize}

\end{theorem}

Furthermore, the distance between $C_m^{(n)}(\xi)$ and $Y_m(\xi)$ can be bounded from above (see for example \cite{MR1177897}, p. 525). We will give here the following bound (see \cite{benarousdang}, p. 15):

\begin{lemma}
\label{lem_bound_Feller}
For any $\theta>0$ there exists a constant $K(\theta)$ depending on $\theta$, such that for every $1\leq m\leq n$,
            \begin{equation}
              \Et{\left|C_m^{(n)}(\xi)-Y_m(\xi)\right|}\leq  \frac{K(\theta)}{n}+\frac \theta n\Psi_n(m),
              \end{equation}
              where

              \begin{equation}
              \Psi_n:=\binom{n-m+\theta-1}{n-m}\binom{n+\theta-1}{n}^{-1}.
              \end{equation}
\end{lemma}
Note that $\Psi_n$ satisfies the following equality:

\begin{lemma}\label{lem_bound_Feller_2}
For each $\theta>0$, there exist some constants $K_1= K_1(\theta)$ and $K_2 = K_2(\theta)$ such that
\begin{equation}
\Psi_n(m)\leq\begin{cases} K_1(1-\frac mn)^{\theta-1} & \text{for } m<n, \\ K_2 n^{1-\theta} & m=n. \end{cases}
\end{equation}
\end{lemma}

The proof of this lemma is straightforward and we thus omit it.
%
%
%
%
%
%
%
%

\subsection{Uniformly distributed sequences}
\label{sec_uniform_dist_seq}

We introduce in this section uniformly distributed sequences and some of their properties. Most of this section is well-known. The only new result is Theorem~\ref{thm_koksma_inequality_2}, which is an extension of the Koksma-Hlawka inequality. For the other proofs (and statements), see the books by Drmota and Tichy \cite{MR1470456} and by Kuipers and Niederreiter \cite{kuipers-niederreiter-74}. 
%

We begin by giving the definition of uniformly distributed sequences.
\begin{definition}
\label{def_d_uniform_dist_in_[0,1]^d}
Let $\vB =\left(\varphi^{(m)}\right)_{m=1}^\infty$ be a sequence in $[0,1]^d$.
For $\overline{\alpha} = (\alpha_1,\dots,\alpha_d)\in [0,1]^d$, we set
\begin{align}
A_n(\overline{\alpha}) = A_n(\overline{\alpha},\vB) := \# \set{1\leq m \leq n; \varphi_m\in [0,\alpha_1]\times \cdots \times [0,\alpha_d] }.
\end{align}
The sequence $\vB$ is called uniformly distributed in $[0,1]^d$ if we have
\begin{align}
\lim_{n\to\infty}
\left|\frac{A_n(\overline{\alpha})}{n}-\prod_{j=1}^d \alpha_j \right|
=
0
\text { for any } \overline{\alpha} \in [0,1]^d.
\end{align}

\end{definition}
%
%
The following theorem shows that the name uniformly distributed is well chosen.
\begin{theorem}
\label{thm_equidist_integral_convergence}
Let $h: [0,1]^d \to \C$ be a Riemann integrable function and $\vB = \left(\varphi^{(m)}\right)_{m\in\N}$
be a uniformly distributed sequence in $[0,1]^d$. Then

\begin{align}
\lim_{n\to\infty}\frac{1}{n} \sum_{m=1}^n h(\varphi^{(m)}) =\int_{[0,1]^d} h(\overline{\phi})  \ d\overline{\phi},
\label{eq_thm_equidist_integral_convergence}
\end{align}
where $d\overline{\phi}$ is the $d$-dimensional Lebesgue measure.
\end{theorem}
\begin{proof}
See \cite[Theorem~6.1]{kuipers-niederreiter-74}
\end{proof}
Theorem~\ref{thm_equidist_integral_convergence} excludes all improper Riemann integrable function like 
$\log(\varphi)$ or $\varphi^{-\gamma}$ with $0<\gamma<1$.
Indeed if $h(\varphi) = \log(\varphi)$ and $\vB$ contains $0$ or a subsequence converging very fast to $0$, 
then the sum on the right hand side of \eqref{eq_thm_equidist_integral_convergence} becomes infinite or fast growing respectively.
However, under some additional assumptions on the sequence $\vB$, one can show that
Theorem~\ref{thm_equidist_integral_convergence} also holds for the logarithm. 
This is the main topic of this section.

%
%
%
%

%
Next, we introduce the discrepancy of a sequence $\vB$.
%
\begin{definition}
\label{def_discrepcy}
Let $\vB = \left(\varphi^{(m)} \right)_{m=1}^\infty$ be a sequence in $[0,1]^d$.
The $*-$discrepancy is defined as
\begin{align}
D^*_n &= D^*_n(\vB)
:=
\sup_{\overline{\alpha} \in [0,1]^d}
\left| \frac{A_n(\overline{\alpha})}{n} - \prod_{j=1}^d \alpha_j  \right|.
\label{eq_def_discrep_2}
\end{align}
\end{definition}
%

%
By the following lemma, Theorem~\ref{thm_equidist_integral_convergence}, the discrepancy and uniformly distributed sequences are closely related. 
\begin{lemma}
\label{lem_uniform_equiv_discrepancy}
Let $\vB=\left(\varphi^{(m)}\right)_{m=1}^\infty$ be a sequence in $[0,1]^d$. The following statements are equivalent:
\begin{enumerate}
 \item $\vB$ is uniformly distributed in $[0,1]^d$.
 \item $\lim_{n\to\infty} D_n^*(\vB)=0$.
 \item Let $h:[0,1]^d \to \C$ be a proper Riemann integrable function. Then
$$
\frac{1}{n} \sum_{m=1}^n h(\varphi^{(m)})  \to \int_{[0,1]^d} h(\overline{\phi})\ d\overline{\phi} \ \text{ for }n\to\infty.
$$
\end{enumerate}
\end{lemma}

The discrepancy allows us to estimate the rate of convergence in Theorem~\ref{thm_equidist_integral_convergence}.

We need as next functions of bounded variation. The definition of bounded variation in the sense of Vitali can 
be found for instance in \cite[Chapter~2.5]{kuipers-niederreiter-74}. 
This definition is slightly technical, but if a function $h:[0,1]^d\to\R$ is enough differentiable, then this reduces to
\begin{align*}
 V(h) = \int_{[0,1]^d} \left|\frac{\partial^d h(\overline{\phi})}{\partial\phi_1\dots\partial\phi_d }\right|  \ d\overline{\phi}.
\end{align*}

Our argumentation requires that the function $h$ behaves well on boundary of $h:[0,1]^d\to\R$. We thus introduce

\begin{definition}
\label{def_bounded_variation_hardy_krause}
Let $h:[0,1]^d \to \C$ be a function. We call $h$ of bounded variation in the sense of Hardy and Krause, if $h$ is
of bounded variation in the sense of Vitali and $h$ restricted to each face $F$ of dimension $1,\ldots,d-1$ of $[0,1]^d$
is also of bounded variation in the sense of Vitali.
We write $V(h|F)$ for the variation of $h$ restricted to face $F$.
\end{definition}
\begin{definition}
\label{def_positive_face}
Let $F$ be a face of $[0,1]^d$. We call a face $F$ positive if there exists a sequence $j_1,\cdots,j_k$ in $\{1,\dots,d\}$ s.t. $F = \bigcap_{m=1}^k \set{ s_{j_m} =1}$, with $s_j$, $1 \leq j \leq d$, being the canonical coordinates in $[0,1]^d$.
\end{definition}
\begin{definition}
\label{def_pi_f(t)}
Let $F$ be a face of $[0,1]^d$ and $\vB$ be sequence in $[0,1]^d$. Let $\pi_F(\vB)$ be the projection of the sequence $\vB$ to the face $F$.
We then write $D_n^*(F,\vB)$ for the discrepancy of the projected sequence computed in the face $F$.
\end{definition}
We are now ready to state the following theorem:
\begin{theorem}[Koksma-Hlawka inequality]
\label{thm_koksma_inequality}
Let $h: [0,1]^d\to \C$ be a function of bounded variation in the sense of Hardy and Krause.
Let $\vB=\left(\varphi^{(m)}\right)_{m\in\N}$ be an arbitrary sequence in $[0,1]^d$. Then
\begin{align}
\left|
\frac{1}{n} \sum_{m=1}^n h(\varphi^{(m)}) - \int_{[0,1]^d} h(\overline{\phi})  \ d\overline{\phi}
\right|
\leq
\sum_{k=1}^d \sum_{\substack{F \text{ positive}\\ \textrm{dim}(F)= k}} D_n^*\bigl(F,\vB) V(h|F)
\label{eq_thm_koksma_inequality}
\end{align}
\end{theorem}
%
\begin{proof}
 See \cite[Theorem 5.5]{kuipers-niederreiter-74}.
\end{proof}
We will consider in this paper only functions of the form
\begin{align}
 h(\overline{\phi})=h(\phi_1,\dots,\phi_d) = \prod_{j=1}^d \log\bigl( f_j( e^{2\pi i\phi_j}) \bigr),
 \label{eq_used_h(s)}
\end{align}
with $f_j$ being piecewise real analytic. In the context of the characteristic polynomial, we will choose $f_j(\phi_j) = |1-e^{2 \pi i \phi_j}|$. Unfortunately, we cannot apply Theorem~\ref{thm_koksma_inequality} in this case, since $\log\bigl|1-e^{2\pi i\phi_j}  \bigr|$ is not of bounded variation.
We thus reformulate Theorem~\ref{thm_koksma_inequality}.
In order to do this, we follow the idea in \cite{HKOS} and \cite{EJP2010-34} and replace $[0,1]^d$ by a slightly smaller set $Q$ such that $\vB \subset Q$ and $h|Q$ is of bounded variation in the sense of Hardy and Krause.\\
We begin with the choice of $Q$. Considering \eqref{eq_used_h(s)}, it is clear that the zeros of $f_j$ cause problems. Thus, we choose $Q$ such that $f_j$ stays away from the zeros ($1\leq j\leq d$).\\
Let $a_{1,j}< \cdots< a_{k_j,j}$ be the zeros of $f_j$ and define $a_{0,j}:=0$ and $a_{k_j+1,j} = 1$ (for $1\leq j \leq d$).
%
We then set for sufficiently small $\delta>0$
\begin{align*}
Q:= \bigcup_{\overline{q}\in \N^d} Q_{\overline{q}}
\ \text{ with } \
 Q_{\overline{q}}:= \prod_{j=1}^d \left[a_{q_j,j}+\delta, a_{q_j+1,j}-\delta\right]
\ \text{ and }\
\widetilde{Q}_{\overline{q}}:= \prod_{j=1}^d \left[a_{q_j,j}, a_{q_j+1,j}\right]
\end{align*}
Note that $\overline{q}=(q_1,\dots,q_d)\in\{0,\dots,k_1+1\}\times\{0,\dots, k_2+1\}\times\cdots\times\{0,\dots,k_d+1\}$ and we consider $Q_{\overline{q}}$ as empty if we have $q_j>k_j+1$ for some $j$.
An illustration of possible $Q$ is given in Figure~\ref{fig_cube_Q}.\\

We will now adjust the Definitions~\ref{def_bounded_variation_hardy_krause},~\ref{def_positive_face} and~\ref{def_pi_f(t)}.
The modification of Definition~\ref{def_bounded_variation_hardy_krause} is obvious.
One simple takes $h$ to be of bounded variation in the sense of Hardy and Krause in each $Q_{\overline{q}}$.
The modification of Definition~\ref{def_positive_face} is also straightforward.
We call a face $F$ of $ Q$ positive if there exists a $\overline{q}\in \N^d$ and a sequence $j_1,\cdots,j_k$ in $\{1,\dots,d\}$ such that, for $s_j$ ($1\leq j\leq d$) being the canonical coordinates in $[0,1]^d$,
\begin{align*}
 F =  \bigcap_{m=1}^k \left( \set{ s_{j_m} = a_{q_{j_m+1},{j_m}}-\delta } \cap Q_{\overline{q}} \right).
\end{align*}

\begin{figure}[h!]
\centering
  \includegraphics[width=.4\textwidth]{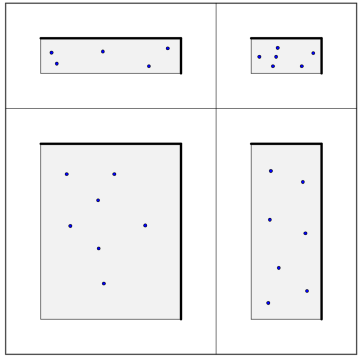}\\
  \caption{Illustration of $Q$, positive faces are bold}\label{fig_cube_Q}
\end{figure}

The modification of Definition~\ref{def_pi_f(t)} is slightly more tricky.
Let $F$ be a face of some $Q_{\overline{q}}$. Let $\vB\cap Q_{\overline{q}}$ be the subsequence of $\vB$ contained in $Q_{\overline{q}}$ and $\pi_F(\vB\cap Q_{\overline{q}})$ be the projection of $\vB\cap Q_{\overline{q}}$ to the face $F$.
Unfortunately we cannot directly compute the discrepancy in the face $F$.
We will see in the proof of Theorem~\ref{thm_koksma_inequality_2} that we have to ``extend $F$ to the boundary of $\widetilde{Q}_{\overline{q}}$''.
More precisely this means to following: We set $\widetilde{F} := L \cap \widetilde{Q}_{\overline{q}}$, where $L$ is the linear subspace generated by $F$ such that $\dim(L) = \dim(F)$ (see Figure~\ref{fig_compute_discrepanz} for an illustration).
The discrepancy $D^*_n(F,\vB)$ is then defined as the discrepancy of $\pi_F(\vB\cap Q_{\overline{q}})$ computed in $\widetilde{F}$.
\begin{figure}[h!]
\centering
  \includegraphics[width=.4\textwidth]{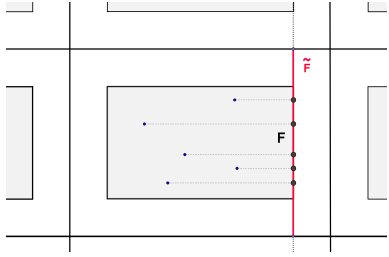}
  \caption{Illustration of $\widetilde{F}$}\label{fig_compute_discrepanz}
\end{figure}

%
%
%
%
%
%
%
%
%
We are now ready to state an extended version of Theorem~\ref{thm_koksma_inequality}. 
\begin{theorem}
\label{thm_koksma_inequality_2}
Let $\delta>0$ be fixed and $\vB = (\varphi^{(m)})_{m=1}^n$ be a sequence in $Q$.
Let $h: Q \to \C$ be a function of bounded variation in the sense of Hardy and Krause.
We then have
\begin{align}
\label{eq_koksma_inequality_2}
\left|
 \frac{1}{n} \sum_{m=1}^n h(\varphi^{(m)}) - \int_{Q} h(\overline{\phi})  \ d\overline{\phi}
\right|
\leq & \
\sum_{k=0}^{d-1} \delta^{d-k} \sum_{\substack{F \\ \textrm{dim}(F)= k}} \int_F h(\overline{\phi}) \ dF
\\
&+
\sum_{k=1}^d \sum_{\substack{F \text{ positive}\\ \textrm{dim}(F)= k}} D_n^*\bigl(F,\vB\bigr) V(h|F).
\nonumber
\end{align}
where $dF=dF(\overline{\phi})$ denotes the Lebesgue measure on the face $F$.
\end{theorem}
\begin{proof}[Proof for $d=1$ and $d=2$]
We assume that $Q=[\delta, 1-\delta]^d$. The more general case can be proven in the same way.\\
The idea is to modify the proof of Theorem~\ref{thm_koksma_inequality} in \cite{kuipers-niederreiter-74}.
There are indeed only minor modifications necessary. We present here only the cases $d=1$ and $d=2$ since we only need these two cases.\\
\underline{$d=1$:} We consider the integral $I_1= I_1(h):=\int_{\delta}^{1-\delta} \bigl(\frac{A_n(\phi)}{n}-\phi\bigr) \ dh(\phi)$ with $A_n(\phi)$ given as in Definition~\ref{def_d_uniform_dist_in_[0,1]^d}.

It is clear from the definition of $D_n^*(\vB)$ that
\begin{equation}
 \left| \int_{\delta}^{1-\delta} \left(\frac{A_n(\phi)}{n}-\phi\right) \ dh(\phi) \right|
\leq  D_n^*(\vB) \int_{\delta}^{1-\delta} \ |dh(\phi)| =  D_n^*(\vB)V(h|[\delta,1-\delta]).
\end{equation}

On the other hand, one can use partial integration and partial summation to show that
\begin{equation}
 I_1
 =
 \bigl(\delta h(1-\delta) +\delta h(\delta)\bigr) + \int_{\delta}^{1-\delta} h(\phi)\ d\phi - \frac{1}{n}\sum_{m=1}^n h(\varphi^{(m)}).
\label{eq_proof_kos_d=1}
\end{equation}
This proves the theorem for $d=1$.\\
\underline{$d=2$:} In this case we consider the integral 
$$I_2=\int_{[\delta,1-\delta]^2} \left(\frac{A_n(\phi)}{n}-\phi_1\phi_2\right) \ dh(\phi_1,\phi_2).$$ 
The argumentation is similar to the case $d=1$. As above, it is immediate that $I_2$ is bounded 
by $D_n^*(\vB)V(h|[\delta,1-\delta]^2)$. On the other hand, we get after consecutive partial integration
\begin{eqnarray}
\label{eq_proof_kok_smooth_part}
 &&\int_{[\delta,1-\delta]^2} \phi_1\phi_2\ dh(\phi_1,\phi_2)\nonumber\\
 &&\quad=\sum_{k=0}^2 \sum_{\substack{F\\ \textrm{dim}(F)= k}} \delta^{2-k} \int_F h \ dF
 -
 \sum_{\substack{F \text{ positive}\\ \textrm{dim}(F)= 1}} \int_F h \ dF\nonumber\\
 &&\qquad+ h(1-\delta,1-\delta) -2\delta h(1-\delta,1-\delta) - \delta h(\delta,1-\delta) -\delta h(1-\delta,\delta)\qquad\quad
\end{eqnarray}
and with two times partial summation
\begin{align}
& \frac{1}{n}\int_{[\delta,1-\delta]^2} A_n(\phi_1,\phi_2) \ dh(\phi_1,\phi_2)\nonumber\\
&\qquad=
 h(1-\delta,1-\delta)
-\sum_{\substack{F \text{ positive}\\ \textrm{dim}(F)= 1}} \frac{1}{n} \sum_{m=1}^n h\bigl( \pi_F(\varphi^{(m)})  \bigr)
+ \frac{1}{n} \sum_{m=1}^n h(\varphi^{(m)}).
\label{eq_proof_kok_A_part}
\end{align}
We now subtract \eqref{eq_proof_kok_smooth_part} from \eqref{eq_proof_kok_A_part} and expand the sum over the positive faces (with $\varphi^{(m)} = (\varphi_1^{(m)}, \varphi_2^{(m)})$). We get
\begin{align}
 I_2
=&
\left(\frac{1}{n} \sum_{m=1}^n h(\varphi^{(m)}) - \int_{Q} h(\overline{\phi})  \ d\overline{\phi}  \right)
-\sum_{k=0}^1 \sum_{\substack{F \\ \textrm{dim}(F)= k}} \delta^{2-k} \int_F h \ dF \label{eq_proof_kok_2_I2}\\
&+
\left( \int_{\delta}^{1-\delta} h(u,1-\delta)\ du - \frac{1}{n} \sum_{m=1}^n h(\varphi_1^{(m)},1-\delta) + \delta h(\delta,1-\delta) +  \delta h(1-\delta,1-\delta)  \right)
\label{eq_proof_kok_2a}\\
&+
\left( \int_{\delta}^{1-\delta} h(1-\delta,v) \ dv - \frac{1}{n} \sum_{m=1}^n h(1-\delta,\varphi_2^{(m)}) + \delta h(1-\delta,\delta) +  \delta h(1-\delta,1-\delta)  \right).
\label{eq_proof_kok_2b}
\end{align}

The brackets \eqref{eq_proof_kok_2a} and \eqref{eq_proof_kok_2b} agree with \eqref{eq_proof_kos_d=1}
if we set ``$h(s) = h(s,1-\delta)$" in \eqref{eq_proof_kok_2a}, respectively "$h(s)= h(1-\delta,s)$" in \eqref{eq_proof_kok_2b}.
We thus can interpret the brackets \eqref{eq_proof_kok_2a} and \eqref{eq_proof_kok_2b} as integrals over the positive faces of $Q$ and apply the induction hypothesis ($d=1$). A simple application of the triangle inequality proves the theorem for $d=2$.\\
It is important to point out that the discrepancy of $(\varphi_1^{(m)})_{m=1}^n$ and $(\varphi_2^{(m)})_{m=1}^n$ is computed in $[0,1]$ and not in $[\delta,1-\delta]$. This observation is the origin for the definition of $D_n^*(F,\vB)$ before Theorem~\ref{thm_koksma_inequality_2}.

\end{proof}

In Section~\ref{sec_CLT_in_ddim}, we will consider sums of the form
\begin{align}
\label{eq:relevent_sums}
 \frac{1}{n} \sum_{m=1}^n \log\left( f_j\left( e^{2\pi i m\varphi_j}\right)\right)    \log\left(f_\ell\left(e^{2\pi i m\varphi_\ell}\right) \right).
\end{align}
We are thus primary interested in ($d$-dimensional) sequences $\vB_{\text{Kro}}=(\varphi_{\text{Kro}}^{(m)})_{m=1}^\infty$, for given $\overline{\varphi}=(\varphi_1,\cdots, \varphi_d) \in \R^d$, defined as follows:
\begin{align}
  \varphi_{\text{Kro}}^{(m)} = \left( \fracpart{m \varphi_1},\dots, \fracpart{m \varphi_d} \right),
\end{align}
where $\fracpart{s}:=s-\intpart{s} \text{ and } \intpart{s}:=\max\set{n\in\Z, n\leq s}$. The sequence $\vB=\vB_{\text{Kro}}$ is called  \emph{Kronecker-sequence of} $\overline{\varphi}$.
%
The next lemma shows that the Kronecker-sequence
is for almost all $\overline{\varphi}\in \R^d$ uniformly distributed.
\begin{lemma}
\label{lem_all_irrational_uniform_distributed}
Let $\overline{\varphi}=(\varphi_1,\dots, \varphi_d)\in \R^d$ be given.
The Kronecker-sequence of $\vB$ is uniformly distributed in $[0,1]^d$ if and only if $1,\varphi_1,\dots,\varphi_d$ are linearly independent over $\Z$.
\end{lemma}
\begin{proof}
 See \cite[Theorem 1.76]{MR1470456}
\end{proof}

Our aim is to apply Theorem~\ref{thm_koksma_inequality} and Theorem~\ref{thm_koksma_inequality_2} for Kronecker sequences. We thus have to estimate the discrepancy in this case and find a suitable $\delta>0$.
We start by giving an upper bound for the discrepancy.
\begin{lemma}
\label{lem_estimate_Dn_Kronecker}
Let $\overline{\varphi} = (\varphi_1,\dots,\varphi_d)\in [0,1]^d$ be given with $1,\varphi_1,\dots,\varphi_d$ linearly independent over $\Z$.
Let $\vB$ be the Kronecker sequence of $\overline{\varphi}$. We then have for each $H\in\N$
\begin{align}
D_n^*(\vB)
\leq
3^d \left(\frac{2}{H+1} + \frac{1}{n} \sum_{0 < \norminf{\overline{q}} \leq H} \frac{1}{r(\overline{q}) \norm{\overline{q}\cdot \overline{\varphi}}}  \right)
\end{align}
with $\norminf{.}$ being the maximum norm, $\norm{a}:= \inf_{n\in\Z} |a-n|$ and $r(\overline{q})=\prod_{i=1}^d \max\set{1,q_i}$ for $\overline{q}=(q_1,\cdots,q_d)\in\N^d$.
\end{lemma}
\begin{proof}
 The proof is a direct application of the Erd\"os-Tur\'an-Koksma inequality (see \cite[Theorem~1.21]{MR1470456}).
\end{proof}

It is clear that we can use Lemma~\ref{lem_estimate_Dn_Kronecker} to give an upper bound for the discrepancy, if we can find a lower bound for $\norm{\overline{q}\cdot \overline{\varphi}}$.
The most natural is thus to assume that $\overline{\varphi}$ fulfills some diophantine inequality.
In order to state this more precise, we give the following definition:
\begin{definition}\label{def_finite_type}
  Let $\overline{\varphi}\in [0,1]^d$ be given.
  We call $\overline{\varphi}$ of finite type if there exist constants $K>0$ and $\gamma \geq 1$ such that
  \begin{align}
    \norm{\overline{q}\cdot \overline{\varphi}} \geq \frac{K}{(\norminf{\overline{q}})^\gamma} \ \text{ for all }\overline{q}\in \Z^d\setminus\set{0}.
  \end{align}
\end{definition}

If $\overline{\varphi} = (\varphi_1,\dots,\varphi_d)$ is of finite type, then it follows immediately from the definition that each $\varphi_j$ is also of finite type and the sequence $1,\varphi_1,\dots,\varphi_d$ is linearly independent over $\Z$.

One can now show the following:
\begin{theorem}\label{thm_estimate_Dn_finite_type}
  Let $\overline{\varphi}\in [0,1]^d$ be of finite type and $\vB$ be the Kronecker sequence of $\overline{\varphi}$. Then
 \begin{align}
    D_n^*(\vB) = O(n^{-\alpha}) \ \text{ for some }\alpha>0.
  \end{align}
\end{theorem}
\begin{proof}
  This theorem is a direct consequence of Lemma~\ref{lem_estimate_Dn_Kronecker} and a simple computation.
  Further details can be found in \cite[Theorem~1.80]{MR1470456} or in \cite{MR2352304}.
\end{proof}
As already mentioned above, we will consider in Section~\ref{sec_CLT_in_ddim} sums of the form \eqref{eq:relevent_sums}.
Surprisingly, it is not necessary to consider summands with more than two factors,
 even when we study the joint behavior at more than two points. We thus give the following definition:
\begin{definition}
\label{def_finite_type_pairwise}
  Let $x_1=e^{2\pi i \varphi_1},\dots, x_d = e^{2\pi i \varphi_d}$ be given. We call both sequences $(x_j)_{j=1}^d$ and $(\varphi_j)_{j=1}^d$ pairwise of finite type, if we have for all $j \neq \ell$ that $(\varphi_j,\varphi_\ell) \in [0,1]^2$ is of finite type in the sense of Definition~\ref{def_finite_type}.
\end{definition}

\section{Central Limit Theorems for the Symmetric Group}
\label{sec_CLT_on_Sn}

In this section, we state general Central Limit Theorems (CLT's) on the symmetric group. 
These theorems will allow us to prove CLT's for the logarithm of the characteristic polynomial and for multiplicative class functions.

For a permutation $\sigma\in S_n$, chosen with respect to the Ewens distribution with parameter $\theta$, let $C_m$ be the random variable corresponding to the number of cycles of length $m$ of $\sigma$. In order to state the CLT's on the symmetric group, we introduce random variables
\begin{align}
\label{eq_def_An}
A_n:=\sum_{m=1}^n\sum_{k=1}^{C_m}X_{m,k},
\end{align}
where we consider $X_{m,k}$ to be independent real valued random variables with
$X_{m,k} \stackrel{d}{=} X_{m,1}$,  for all $1\leq m\leq n$ and $k\geq 1$. Furthermore, all $X_{m,k}$ are independent of $\sigma$. 
Of course, if $X_{m,k}=\Re(\log (1-x^{-m} T_{m,k}))$ (or $\Im(\log (1-x^{-m} T_{m,k}))$), then $A_n$ is equal in law to the real (or imaginary) part of $\log Z_{n,z}(x)$, which is the logarithm of the characteristic polynomial of $M_{\sigma, z}$. This will be treated in Section~\ref{sec_examples}.

\subsection{Degenerate case}

We give in this subsection an overview over degenerate case $X_{m,k} \equiv a_m$ with $a_m\in\mathbb{R}$.
The second author has proven for this situation in \cite{DZ-CLT} a central limit theorem for $A_n$ with a Lyapunov condition using the Feller-coupling.
A more modern approach base on generating functions and complex analysis.
This method has been used by Manstavi{\v{c}}ius in \cite{Ma96} to prove a central limit theorem for $A_n$ with a Lindeberg-Feller condition.
Furthermore Manstavi{\v{c}}ius has given in \cite{Ma08} sufficient and necessary conditions for the weak convergence of a sightly more general random variables and 
Babu and Manstavi{\v{c}}ius have extended in \cite{BaMa99} the CLT to a functional limit theorem. An overview can be found in \cite{Ma11} and in the references therein.

\subsection{One dimensional CLT}

The argumentation by Manstavi{\v{c}}ius can also be used in the situation for non degenerate $X_{m,k}$ and
to extend the CLT to weighted measure recently studied by Ercolani and Ueltschi \cite{ErUe11} 
(Details about the weighted measure on the symmetric group can be found for instance in \cite{HuNaNiZe11}, \cite{MaNiZe11}, \cite{NiStZe13}, \cite{NiZe11}).
This computations are quit involed. We thus postpone them to a further paper and use instead the following CLT proven in \cite{MaNiZe11}

\begin{theorem}[Hughes, Nikeghbali, Najnudel, Zeindler (2011)]
\label{thm_trace_infinite_case}
Assume that
\begin{align}\label{eq:V_N}
    V_n := \sum_{m=1}^n \frac{1}{m} \E{\left(X_{m,1}\right)^2} \longrightarrow \infty \qquad (n\to\infty).
  \end{align}
Assume further that there exists a $p>\max\left\{ \frac{1}{\theta} , 2 \right\}$ such
        that
\begin{equation}\label{eq:inf_var_cond}
\sum_{m=1}^{n} \frac{1}{m}  \E{\left|X_{m,1}\right|^p} =
        o\left(V_n^{p/2}\right)
\end{equation}
Then
  \begin{align}
    \left( \frac{A_n - \E{A_n}}{ \sqrt{\theta V_N}} \right)_{n \geq 1}
  \end{align}
converges in distribution to a standard gaussian random
variable.
%
\end{theorem}

\begin{proof} This theorem can be obtained immediately from Theorem 6.2 in \cite{MaNiZe11} by setting $X_{m,k}= k \Delta_k$.
For completeness, we give a short overview over the proof.
The proof base on the Feller coupling (see Section~\ref{sec_Feller_coupling}).
This ensures that the random variables $C_m$ and $Y_m$ are defined on the same space and can be compared with Lemma~\ref{lem_bound_Feller}.
The strategy of the proof is the following: define

\begin{align}
 B_n=\sum_{m=1}^n\sum_{k=1}^{Y_m}X_{m,k},
\end{align}
and show that $A_n$ and $B_n$ have the same asymptotic behavior after normalization.
This can be done for instance by showing that $\E{|A_n-B_n|} = O(1)$. We have
\begin{eqnarray}
&&\E{|A_n-B_n|}
=
\E{\left|\sum_{m=1}^n \left(\sum_{k=1}^{C_m}X_{m,k} - \sum_{k=1}^{Y_m} X_{m,k} \right) \right|}\nonumber\\
&&\quad
\leq
\sum_{m=1}^n\E{\left|\sum_{k=(C_m\wedge Y_m)+1}^{C_m\vee Y_m}X_{m,k}\right|}
\leq
\sum_{m=1}^n\E{\sum_{k=(C_m\wedge Y_m)+1}^{C_m\vee Y_m} \E{\left| X_{m,k}\right|}  }\nonumber\\
&&\quad\leq
\sum_{m=1}^n \E{|X_{m,1}|}\E{\left|{C_m}-{Y_m}\right|}\qquad
\end{eqnarray}
By Lemma~\ref{lem_bound_Feller}, there exists for any $\theta>0$ a constant $K(\theta)$, such that
\begin{equation}\label{eq_proof_aux_thm_1}
 \sum_{m=1}^n\E{|X_{m,1}|}\E{|C_m-Y_m|}
 \leq
 \frac {K(\theta)}n\sum_{m=1}^n\E{|X_{m,1}|}+\frac \theta n\sum_{m=1}^n\Psi_n(m)\E{|X_{m,1}|}.
\end{equation}
One now can show with the H\"older inequality, Lemma~\ref{lem_bound_Feller_2} and the assumptions of the theorem that this quantity
is indeed $O(1)$. It is thus enough to consider only $B_n$, but $B_n$ is just a sum of independent random variables and the 
theorem follows from the Lyapunov CLT.
\end{proof}

\subsection{Multi dimensional central limit theorems}

In this section, we replace the random variables $X_{m,k}$ in Theorem~\ref{thm_trace_infinite_case} by $\R^d$-valued random variables $\overline{X}_{m,k}=(X_{m,k,1},\dots,X_{m,k,d})$ and prove a CLT for
\begin{equation}
\overline{A}_{n,d}
:=
\sum_{m=1}^n\sum_{k=1}^{C_m}\overline{X}_{m,k}.
\end{equation}
As before, we assume that $\overline{X}_{m,k}$ is a sequence of independent random variables such that $\overline{X}_{m,k} \stackrel{d}{=} \overline{X}_{m,1}$ and all $\overline{X}_{m,k}$ and $\sigma\in S_n$ are independent. We will prove the following theorem:

\begin{theorem}
\label{thm_general_CLT_ddim}
Assume there exists constants $V_n$ with $V_n \to \infty$ as $n\to\infty$ 
and there exists constants $\sigma_{j,\ell}$ such that for all $1\leq j,\ell \leq d$
  %
%
\begin{align}\label{eq:V_N_ddim}
 \sum_{m=1}^n \frac{1}{m} \E{X_{m,1,j}X_{m,1,\ell}} \sim \sigma_{j,\ell} \cdot V_n \qquad (n\to\infty).
\end{align}
Assume further that there exists a $p>\max\left\{ \frac{1}{\theta} , 2 \right\}$ such
        that for each $1\leq j \leq d$
\begin{equation}\label{eq:inf_var_cond_ddim}
\sum_{m=1}^{n} \frac{1}{m}  \E{\left|X_{m,1,j}\right|^p} =
        o\left(\bigl(V_n\bigr)^{p/2}\right)
\end{equation}
Then the distribution of
\begin{align}
\frac{\overline{A}_{n,d}-\E{\overline{A}_{n,d}}}{\sqrt{\theta\,V_n}}
\end{align}
converges in law to the normal distribution $\mathcal{N}(0,\Sigma)$,
where $\Sigma$ is the covariance matrix $(\sigma_{i,j})_{1\leq i,j\leq d}$.
\end{theorem}
%
%
%
%
%

\begin{proof}
The theorem follows from the Cramer-Wold theorem if we can show for each $\overline{t}=(t_1,\dots, t_d)\in\R^d$
\begin{align}
 \overline{t}\cdot \frac{\overline{A}_n-\E{\overline{A}_n}}{\sqrt{V_n}}
 \stackrel{d}{\longrightarrow} \mathcal{N}(0,\theta\overline{t}\Sigma\overline{t}^T).
\end{align}
A simple computation shows that
\begin{align}
 \overline{t} \cdot \overline{A}_n
 =
 \sum_{m=1}^n\sum_{k=1}^{C_m} H_{m,k}^{(d)}
 \end{align}
with
\begin{equation}
H_{m,k}^{(d)}:=H_{m,k}=\sum_{j=1}^dt_jX_{m,k,j}.
\end{equation}
We now show that $H_{m,k}$ fulfills the conditions of Theorem~\ref{thm_trace_infinite_case}.
Clearly, $H_{m,k}$ is a sequence of independent random variables, $H_{m,k} \stackrel{d}{=}H_{m,1}$ and $H_{m,k}$ is independent of $C_b$ for all $m,k,b$.
We get 
\begin{align}
\sum_{m=1}^n \frac{1}{m} \E{H_{m,1}^2}
  &=
 \sum_{m=1}^n \frac{1}{m} \sum_{j,\ell =1}^d t_j t_\ell
  \E{X_{m,1,j}X_{m,1,\ell}}\nonumber\\
  &\sim
  V_n\sum_{j,\ell =1}^d t_j t_\ell \sigma_{j,\ell}
  =
  V_n \cdot\overline{t}\Sigma\overline{t}^T
\end{align}
with $\Sigma = (\sigma_{j,\ell})_{1 \leq j,\ell \leq d}$. This shows that \eqref{eq:V_N} is fulfilled.
%
%
We now look at \eqref{eq:inf_var_cond}. We use that $|x+y|^p \leq 2^{p-1}(|x|^p+|y|^p)$ for $p\geq 1$ and get
\begin{align*}
 &\sum_{m=1}^n \frac{1}{m}\E{|H_{m,1}|^p}
  \leq K_p
  \sum_{m=1}^n \frac{1}{m}\E{ \sum_{j=1}^d |t_j|^p|X_{m,1,j}|^p}
 =
 o\left(\bigl(V_n\bigr)^{p/2}\right)
 \end{align*}
where $K_p$ depends only on $p$ and $d$. This concludes the proof of Theorem~\ref{thm_general_CLT_ddim}.
\end{proof}

{\bf Remark:}
It is clear that Theorem~\ref{thm_general_CLT_ddim} can be used for complex random variables, by identifying $\C$ by $\R^2$.

\section{Results on the Characteristic Polynomial and Multiplicative Class Functions}
\label{sec_examples}

In this section we apply the theorems in Section~\ref{sec_CLT_on_Sn} to the characteristic polynomial and multiplicative class functions.
We start by considering in Section~\ref{sec_CLT_in_1dim} the real and imaginary parts separately and give results on the joint behavior 
and the behavior at different points in in Section~\ref{sec_CLT_in_ddim}.

Recall that we study the characteristic polynomial in terms of $Z_{n,z}(x)$ and recall the definitions for the multiplicative class functions $W^{1,n}_z(f)$ and $ W^{2,n}_z(f)$, given by Definitions~\ref{def_W1_poly} and~\ref{def_W2}.
As in Definition~\ref{def_log_Zn,Wn_1dim_1}, it is natural to choose the branch of logarithm as follows:

\begin{definition}
\label{def_log_Zn,Wn_1dim}
Let $x = e^{2 \pi i \varphi} \in \mathbb{T}$ be a fixed number, $z$ a $\mathbb{T}$--valued random variable and $f:\mathbb{T} \to \C$  a real analytic function. Furthermore, let $(z_{m,k})_{m,k=1}^\infty$ and $(T_{m,k})_{m,k=1}^\infty$ be two sequences of independent random variables, independent of $\sigma \in S_n$ with
\begin{align}
  z_{m,k} \stackrel{d}{=} z \ \text{ and } \  T_{m,k} \stackrel{d}{=} \prod_{j=1}^m z_{j,k}.
\end{align}
We then set
\begin{align}
 \label{eq_log_Zn}
  \log\bigl(Z_{n,z}(x) \bigr)
  &:=
  \sum_{m=1}^n\sum_{k=1}^{C_m} \log (1-x^{-m}T_{m,k}),\\
  \label{eq_log_W1}
  w^{1,n}(f)(x)
  :=
  \log\left( W^{1,n}_z(f)(x) \right)
  &:=
  \sum_{m=1}^n\sum_{k=1}^{C_m} \log\bigl( f(x^m z_{m,k} )\bigr),\\
  \label{eq_log_W2}
  w^{2,n}(f)
  :=
  \log\left( W^{2,n}_z(f)(x) \right)
  &:=
  \sum_{m=1}^n\sum_{k=1}^{C_m} \log\bigl( f(x^mT_{m,k})\bigr).
\end{align}
%
\end{definition}
%

%
%

\subsection{Limit behavior at $1$ point}
\label{sec_CLT_in_1dim}

The following results are important cases for which the conditions in Theorem~\ref{thm_trace_infinite_case} are satisfied. We will show the following central limit theorem results for multiplicative class functions.

\begin{theorem}
\label{thm_CLT_W1_1dim_general}
Let $S_n$ be endowed with the Ewens distribution with parameter $\theta$, $f$ be a non zero real analytic function, $z$ a $\mathbb{T}$-valued random variable and $x = e^{2\pi i \varphi} \in \mathbb{T}$ be not a root of unity, i.e. $x^m \neq 1 $ for all $m\in\Z$.

Suppose that one of the following conditions is satisfied,

\begin{itemize}
  \item $z$ is uniformly distributed,
  \item $z$ is absolutely continuous with bounded, Riemann integrable density,
  \item $z$ is discrete, there exists a $\rho>0$ with $z^\rho \equiv 1$, all zeros of $f$ are roots of unity and $x$ is of finite type (see Definition~\ref{def_finite_type}).
\end{itemize}

Then,
   \begin{align}
  \label{eq CLT_z_w1_real}
    \frac{\Re\left(w^{1,n}(f)\right)}{\sqrt{\log n}}
    -
    \theta\cdot m_R(f)\sqrt{\log n} 
    \stackrel{d}{\longrightarrow}
     N_R,\\
      \label{eq CLT_z_w1_im}
     \frac{\Im\left(w^{1,n}(f)\right)}{\sqrt{\log n}}
    -
    \theta\cdot m_I(f)\sqrt{\log n}
    \stackrel{d}{\longrightarrow}
     N_I
  \end{align}
  with $N_R\sim \mathcal{N} \left(0,\theta V_R(f) \right), N_I\sim \mathcal{N} \left(0,\theta V_I(f) \right)$ and
  \begin{align}
    m_R(f)
    &=
    \Re\left( \int_{0}^1 \log\bigl(f(e^{2 \pi i \phi}) \bigr) \ d\phi \right), \quad
    V_R(f)
    =
    \int_{0}^1 \log^2\bigl|f(e^{2 \pi i \phi}) \bigr| \ d\phi,\\
    m_I(f)
    &=
    \Im\left( \int_{0}^1 \log\bigl(f(e^{2 \pi i \phi}) \bigr) \ d\phi \right),\quad
    V_I(f)
    =
    \int_{0}^1 \arg^2\bigl(f(e^{2 \pi i \phi}) \bigr) \ d\phi.
  \end{align}
\end{theorem}

\begin{theorem}
\label{thm_CLT_W2_1dim_general}
Let $S_n$ be endowed with the Ewens distribution with parameter $\theta$, $f$ be a non zero real analytic function, $z$ a $\mathbb{T}$-valued random variable and $x \in \mathbb{T}$ be not a root of unity.

Suppose that one of the following conditions is satisfied,

\begin{itemize}
  \item $z$ is uniformly distributed,
  \item $z$ is absolutely continuous with density $g:[0,1] \to \R_+$, such that
\begin{align}
g(\phi)
=
\sum_{j\in\Z} c_j e^{2 \pi i j \phi}
\ \text{ with } |c_j|<1 \text{ for } j\neq 0 \text{ and } \sum_{j\in\Z} |c_j| < \infty.
\end{align}
\item $z$ is discrete, there exists a $\rho>0$ with $z^\rho \equiv 1$, all zeros of $f$ are roots of unity, $x$ is of finite type (see Definition~\ref{def_finite_type}) and for each $1\leq k\leq \rho$, 
\begin{align}
  \Pb{z = e^{2 \pi i k/\rho}} = \frac{1}{\rho} \sum_{j=0}^{\rho-1} c_j e^{2 \pi i jk}
\ \text{ with } |c_j|<1 \text{ for } j\neq0.
\end{align}
\end{itemize}

Then,
   \begin{align}
  \label{eq CLT_z_w2_real}
    \frac{\Re\left(w^{2,n}(f)\right)}{\sqrt{\log n}}
    -
     \theta\cdot m_R(f)\sqrt{\log n}
    \stackrel{d}{\longrightarrow}
     N_R,\\
      \label{eq CLT_z_w2_im}
     \frac{\Im\left(w^{2,n}(f)\right)}{\sqrt{\log n}}
    -
    \theta \cdot m_I(f)\sqrt{\log n} 
    \stackrel{d}{\longrightarrow}
     N_I,
  \end{align}
with $m_R(f), m_I(f), N_R$ and $N_I$ as in Theorem~\ref{thm_CLT_W1_1dim_general}.
\end{theorem}

Note that the uniform case is included in the absolutely continuous case. Furthermore, $Z_{n,z}(x)$ is the special case $f(x)=1-x^{-1}$ of $W^2$. Thus, a direct consequence of Theorem~\ref{thm_CLT_W2_1dim_general} is the following corollary, which, after a short computation, covers Proposition~\ref{prop_Zn_general_1dim}:
\begin{corollary}
\label{cor_Zn_general_1dim}
Let $S_n$ be endowed with the Ewens distribution with parameter $\theta$, $z$ a $\mathbb{T}$-valued
random variable and $x \in \mathbb{T}$ be not a root of unity, i.e. $x^m \neq 1 $ for all $m\in\Z$.

Suppose that one of the conditions in Theorem~\ref{thm_CLT_W2_1dim_general} holds, then
 \begin{align}
  \label{eq CLT_z_general_Zn_real}
    &\frac{\Re\left(\log\bigl(Z_{n,z}(x) \bigr)\right)}{\sqrt{\log n}}
    \stackrel{d}{\longrightarrow}
     N_R \quad\text{and}\\
      \label{eq CLT_z_gerneral_Zn_im}
     &\frac{\Im\left(\log\bigl(Z_{n,z}(x) \bigr)\right)}{\sqrt{\log n}}
    \stackrel{d}{\longrightarrow}
     N_I,
  \end{align}
  with $N_R,N_I \sim \mathcal{N} \left(0,\theta\frac{\pi^2}{12} \right)$.
\end{corollary}

In Corollary~\ref{cor_Zn_general_1dim}, $\Re\left(\log\bigl(Z_{n,z}(x) \bigr)\right)$ and $\Im\left(\log\bigl(Z_{n,z}(x) \bigr)\right)$ are converging to normal random variables without centering. We will see that this is due to the expectation being $o(\sqrt{\log n})$. \\

{\bf Remark:}
  The case $x$ a root of unity can be treated similarly. The computations are indeed much simpler, see for instance \cite{DZ-CLT} for $z\equiv 1$.\\

{\bf Proof of Theorem~\ref{thm_CLT_W1_1dim_general}}
\begin{proof}
It is clear from Definition~\ref{def_log_Zn,Wn_1dim} that the real and imaginary parts of the random variables 
$\log\bigl(Z_{n,z}(x) \bigr)$, $w^{1,n}(f)(x)$ and $w^{2,n}(f)$ have the form \eqref{eq_def_An}. 
We thus can use Theorem~\ref{thm_trace_infinite_case} to study their behaviour as $n\to\infty$.
We show that the assumptions 
of Theorem~\ref{thm_trace_infinite_case} are fulfilled with $V_n\sim \text{const.}\log(n)$ 
in each case considered in Theorem~\ref{thm_CLT_W1_1dim_general}.
For this, we use the following observation: If $(a_m)_{m\in\N}$ is a sequence of complex numbers, then 
\begin{align}
\label{eq:partial_sum}
 \frac{1}{n}\sum_{m=1}^n a_m \to E\ \Longrightarrow \sum_{m=1}^n \frac{a_m}{m} = E\log(n) + O(1) \qquad(n\to\infty).
\end{align}
This statement follows with partial summation and a direct computation. It is thus enough to show that
we have for $p=2$ and some $p>2$ 
\begin{align}
\label{eq:moment_sum_X_m}
\frac{1}{n} \sum_{m=1}^{n}   \E{\left|X_{m,1}\right|^p} \ \longrightarrow  \ E_p 
\end{align}
as $n\to\infty$ with $E_p\in\R$ depending on $p$ and the case studied.

 \ \newline
{\it Uniform measure on the unit circle.}

We start with the simple case where $z$ is uniformly distributed. 
We begin with the real part and put $X_{m,k} = \log|f(x^mz_{m,k})|$. We use that  $x^m z_{m,k} \stackrel{d}{=} z_{m,1}$ for $m$ fix and get
\begin{align}
   \label{eq_z_uniform_w1}
\E{|X_{m,1}|^p}= \E{\bigl|\log|f(x^mz_{m,1})|\bigr|^p}
=
\int_0^1 \bigl|\log\bigl|f(e^{2 \pi i \phi})\bigr|\bigr|^p \ d\phi.
\end{align}
We have to justify that the integral in \eqref{eq_z_uniform_w1} exists. Since $f$ is a non-zero real analytic function, 
we have for $x_0 = e^{2\pi i \phi_0}$ being a zero of $f$, 
\begin{align}
  \log\bigl| f(e^{2\pi i \phi})| \sim K \log|\phi -\phi_0|,
\end{align}
as $\phi \to\phi_0$ and a $K>0$. The integral in \eqref{eq_z_uniform_w1} now exists for each $p\geq 1$ 
since $ \log|\phi -\phi_0|^p$ is integrable in a neighbourhood of $\phi_0$ for each $p\geq 1$ and $f$ has at most finitely many zeros. 

We thus have obviously for each $p\geq 1$
\begin{align}
\label{eq:uniform_moment_p}
\frac{1}{n} \sum_{m=1}^{n}   \E{\left|X_{m,1}\right|^p} = \int_0^1 \bigl|\log\bigl|f(e^{2 \pi i \phi})\bigr|\bigr|^p \ d\phi.
\end{align}
The observation in \eqref{eq:partial_sum} together with \eqref{eq:uniform_moment_p} for 
$p=2$ and any $p>2$ implies that the assumptions Theorem~\ref{thm_trace_infinite_case} are fulfilled with $V_n \sim V_R(f) \log(n)$.
%
%
%
%
It remains to compute the asymptotic behaviour of the expectation.
We use the Feller-coupling (see Section~\ref{sec_Feller_coupling}) and get
%
\begin{align}
  &\E{\Re\left(w^{1,n}(f)\right)}
  =
  \sum_{m=1}^n \E{C_m} \E{\log\bigl|f(x^mz_{m,1})\bigr|}\nonumber\\ 
  &=
  \left(\sum_{m=1}^n \E{Y_m} \E{\log\bigl|f(x^mz_{m,1})\bigr|}\right)
 +  
\left(\sum_{m=1}^n \E{C_m - Y_m} \E{\log\bigl|f(x^mz_{m,1})\bigr|}\right)  
\nonumber\\
  &=
  \left(\int_0^1 \log\bigl|f(e^{2 \pi i \phi})\bigr| \ d\phi \right) \left( \sum_{m=1}^n\frac{\theta}{m} \right) + O\left(\sum_{m=1}^n (\E{C_m}-\E{Y_m}) \right)  
 \nonumber\\
  &=
  \theta\cdot m_R(f) \log n  + O(1).
\end{align}

We have used in last equality the inequalities in Lemma~\ref{lem_bound_Feller} and Lemma~\ref{lem_bound_Feller_2} to obtain the $O(1)$ term.
The computations are straightforward and we thus omit them. This completes the computations for the real part.

Consider now the imaginary part with 
$$X_{m,k}=\Im(\log(f(x^mz_{m,k})))=\arg\bigl(f(x^mz_{m,k})\bigr).$$ 
Obviously, $\arg\bigl( f(e^{2\pi i \phi})\bigr)$ is bounded and piecewise real analytic with at most finitely many discontinuity points as function in $\phi$.
Thus all moments of $X_{m,k}$ exists. We therefore can use precise the same argumentation as for the real part and thus omit this computations.
%

{\it Absolute continuous case.}

We start again with the real part and use as before $X_{m,k}:= \log\bigl| f (z_{m,k}x^m)\bigr|$
with $x = e^{2\pi i \varphi}$. 
For simplicity, we write $h(\phi):=\log\bigl| f(e^{2\pi i \phi}) \bigr|$.
We first show that all moments of $X_{m,k}$ exist. We write $g$ for the density of $z_{m,k}$ and obtain for all $p\geq1$,
\begin{align}
  \E{|X_{m,k}|^p} 
&= 
\int_0^1 \left|\log\bigl| f(x^m e^{2\pi i \phi}) \bigr|\right|^p g(\phi) \ d\phi
=
\int_{0}^1 |h(\phi+m\varphi)|^p g(\phi) \ d\phi
\end{align}
We extend the function $g(\phi)$ periodically to $\R$ and get
\begin{align}
 \E{|X_{m,k}|^p} &= \int_0^1 |h(\phi)|^p g(\phi-m\varphi) \ d\phi \leq
\sup_{\alpha\in[0,1]} |g(\alpha)| \int_{0}^1 |h(\phi)|^p \ d\phi 
< \infty.
\end{align}
This is finite since $g$ is by assumption bounded.
We now show that the assumptions of Theorem~\ref{thm_trace_infinite_case} are satisfied 
by computing the asymptotic behaviour of the expression \eqref{eq:moment_sum_X_m} in this case.
By assumption, $x=e^{2i\pi\varphi}$ is not a root of unity and $\varphi$ is thus irrational.
Therefore, $(\fracpart{m\varphi})_{m=1}^\infty$ is uniformly distributed in $[0,1]$. Since $g$ is Riemann integrable, 
we can apply Theorem~\ref{thm_equidist_integral_convergence} for fixed $\phi$ and obtain as $n\to\infty$
\begin{align}
\label{eq:sum_density}
  \frac{1}{n} \sum_{m=1}^n g(\phi-m\varphi)
  =
  \frac{1}{n} \sum_{m=1}^n g(\phi-\fracpart{m\varphi})
  \longrightarrow
  \int_{0}^1 g(\psi) \ d\psi = 1.
\end{align}
Since $g$ is bounded and $h^p$ is integrable, we can use dominated converge and get
\begin{align}\label{eq_W1_hp_expectation}
\frac{1}{n} \sum_{m=1}^n \E{|X_{m,k}|^p}
=
\int_0^1 |h(\phi)|^p  \left(\frac{1}{n} \sum_{m=1}^n g(\phi-m\varphi) \right)\ d\phi
\to
\int_{0}^1 |h(\phi)|^p\ d\phi.
\end{align}
Thus the assumptions of Theorem~\ref{thm_trace_infinite_case} are fulfilled with $V_n=V_R(f)\log(n)$.
It remains to show that the real part of $\E{w^{1,n}(f)}$ can be replaced by $\theta\cdot m_R(f) \log n$.
This computation is similar and we thus omit it. 
Also the computations for the imaginary part are almost the same as for the real part and can be omitted as well.

{\it Discrete $z$.}
We have $z^\rho\equiv1$ for some $\rho\geq 1$ and thus
\begin{align}
  \ \E{\bigl|\log\bigl| f (z_{m,1}x^m)\bigr|\bigr|^p}
&=
\sum_{k=1}^\rho \Pb{z=e^{2\pi i k/\rho}}  \left( \bigl|\log\bigl| f (e^{2\pi i k/\rho}x^m)\bigr|\bigr|^p \right)\nonumber\\
&=
\sum_{k=1}^\rho \Pb{z=e^{2\pi i k/\rho}} \left|h\left(k/\rho + m\varphi\right)\right|^p
\label{eq_discrete_pth_moment}
\end{align}
This sum is well defined since $x= e^{2\pi i \varphi}$ is by assumption not a root of unity and all zeros of $f$ 
are roots of unity. The computation of the expression \eqref{eq:moment_sum_X_m}
is in this case slightly more difficult. We use
that the sequence $(x^m)_{m\in\N}$ is uniformly distributed and 
show here for each $1\leq k\leq \rho$
\begin{align}
\label{eq:sum_diskrte_2}
\frac{1}{n}\sum_{m=1}^n 
 \left|h\left(k/\rho + m\varphi\right)\right|^p
 \longrightarrow
\int_0^1 \left|h(\phi)\right|^p \ d\phi.
\end{align}
The function $h(\phi)$
is not of bounded variation (except when $f$ is zero-free) and we thus use Theorem~\ref{thm_koksma_inequality_2} for $d=1$. 
%
%
%
We omit the details of this computation since they can be founded in \cite[p.14--15]{DZ-CLT} and 
since we use in Section~\ref{sec_CLT_in_ddim} the same argumentation for $d=2$.
It follows with \eqref{eq_discrete_pth_moment} and \eqref{eq:sum_diskrte_2} that
\begin{align}
\frac{1}{n}\sum_{m=1}^n \E{\bigl|\log\bigl| f (z_{m,1}x^m)\bigr|\bigr|^p}
\to
\int_0^1 \left|h(\phi)\right|^p \ d\phi.
\end{align}
The remaining argumentation is the same as in the previous cases and will be thus omitted.

%
%
%
%


\end{proof}

{\bf Proof of Theorem~\ref{thm_CLT_W2_1dim_general}}
\begin{proof}\ We will use here the same argumentation as in the proof of Theorem~\ref{thm_CLT_W1_1dim_general} and thus
verify only \eqref{eq:moment_sum_X_m} for each case considered. 
We will use again the notation $h(\phi):=\log\bigl| f(e^{2\pi i \phi}) \bigr|$ and $x =e^{2\pi i \varphi}$.

{\it z uniform.}\\
Since $T_{m,k}$ is uniformly distributed, we have $T_{m,k}\stackrel{d}{=} z_{m,k}$ and thus $w^{1,n}(f) \stackrel{d}{=} w^{2,n}(f)$. This case is therefore already proven.


{\it z absolutely continuous}

%

We first consider the real part of $w^{2,n}(f)$, i.e.
\begin{align}
  X_{m,k}:= \log\bigl| f (x^m T_{m,k})\bigr|.
\end{align}
The density of $T_{m,k}$ is $g^{*m}$, where $g^{*m}$
is the $m-$times convolution of $g$ with itself and $g$ is the density of $z$.
We first show that all moments of $X_{m,k}$ exists.
By assumption,%
\begin{align}
\label{eq:assume:g_w2}
g(\phi)
=
\sum_{j\in\Z} c_j e^{2 \pi i j \phi}
\ \text{ with } \ |c_j|\leq 1 \text{ for } j\neq 0
\ \text{ and } \ \sum_{j\in\Z} |c_j| < \infty.
\end{align}
The properties of the Fourier transform immediately imply
\begin{align}
\label{eq:assume:g*_w2}
g^{*m}(\phi)
=
\sum_{j\in\Z} c_j^m e^{2 \pi i j \phi}.
\end{align}
%
As before, we first show that all moments of $X_{m,k}$ are finite. We have
\begin{align}
  \E{|X_{m,k}|^p}
&=
\int_{0}^1 |h(\phi+m\varphi)|^p g^{*m}(\phi) \ d\phi
\leq \sum_{j\in\Z} |c_j|^m \int_{0}^1 |h(\phi)|^p \ d\phi
\nonumber\\
&\leq
\sum_{j\in\Z} |c_j| \int_{0}^1 |h(\phi)|^p \ d\phi
<
\infty
\end{align}
This shows that all moments exists and can be bounded independently of $m$.


We now show that for each $p\geq 1$
\begin{align}
    \frac{1}{n} \sum_{m=1}^n \E{|X_{m,k}|^p}
\to
\int_0^1 |h(\phi)|^p \ d\phi.
%
\end{align}
We have
\begin{align}
  \frac{1}{n} \sum_{m=1}^n \E{|X_{m,k}|^p}
&=
  \frac{1}{n} \sum_{m=1}^n \int_0^1 |h(\phi + m\varphi)|^p g^{*m} (\phi) \ d\phi
\nonumber\\
&=
  \int_{0}^1 |h(\phi)|^p \left(\frac{1}{n} \sum_{m=1}^n g^{*m} (\phi - m\varphi)\right) \ d\phi.
\label{eq_stetig W2_second_moment_1dim}
\end{align}
Consider now $\frac{1}{n} \sum_{m=1}^n g^{*m} (\phi - m\varphi)$  for $\phi$ fix.
We use assumption~\ref{eq:assume:g_w2} together with \eqref{eq:assume:g*_w2} and get
\begin{align}
  \frac{1}{n} \sum_{m=1}^n g^{*m} (\phi -m\varphi)
&=
\frac{1}{n} \sum_{m=1}^n \sum_{j\in\Z} c^m_j e^{2 \pi i j (\phi-m\varphi)}
\nonumber\\
&=
\sum_{j\in\Z}  e^{2 \pi i j \phi} \left(\frac{1}{n} \sum_{m=1}^n c^m_j e^{-2 \pi i j m\varphi}\right).
\label{eq_fourier_W2_eingesetz}
\end{align}
We thus have to compute the behavior of 
\begin{align}
\label{eq:case_w2_absolute_boring}
 \frac{1}{n}\sum_{m=1}^n c^m_j e^{-2 \pi i j m\varphi}.
\end{align}

For $j=0$, this expression is always $1$, since $c_0 = \int_0^1 g(\phi) \ d\phi = 1$. 
For $j \neq 0$, we use the assumption $|c_j|<1$ and get 
\begin{align}
\label{eq:w_2_blla_to 0}
\frac{1}{n} \left|\sum_{m=1}^n c^m_j e^{-2 \pi i j m\varphi}\right|
\leq
\frac 1n \left|\frac{1-c_j^{n+1}e^{2i\pi j\varphi(n+1)}}{1-c_j e^{2i\pi j\varphi}}\right|
\leq 
\frac 1n \frac{2}{1-|c_j|}\to 0 \qquad (n\to\infty).\nonumber\\
\end{align}
%
It is thus to expect that the expression in \eqref{eq_fourier_W2_eingesetz} converges for almost all $\phi$ to $1$.
To verify this, we use dominated convergence. We have $\frac{1}{n} \sum_{m=1}^n |c_j^m| \leq |c_j|$  and thus
\begin{align}
  \frac{1}{n} \sum_{m=1}^n g^{*m} (\phi -m\varphi)
  \leq 
  \left|\sum_{j\in\Z}  e^{2 \pi i j \phi} \left(\frac{1}{n} \sum_{m=1}^n c^m_j e^{-2 \pi i j m\varphi}\right)\right|
\leq
\sum_{j\in\Z} |c_j|  < \infty.
\end{align}
Therefore, as $n\to\infty$ and for almost all $\phi$,
\begin{align}
  \frac{1}{n} \sum_{m=1}^n g^{*m} (\phi -m\varphi)
\longrightarrow 1 \ \ (n\to\infty).
\end{align}
Furthermore, $\sum_{j} |c_j|$ is also an upper bound for $\frac{1}{n} \sum_{m=1}^n g^{*m} (\phi - m\varphi)$. So again, we can use in \eqref{eq_stetig W2_second_moment_1dim} dominated convergence and obtain
\begin{align}
    \frac{1}{n} \sum_{m=1}^n \E{|X_{m,k}|^p}
\to
\int_0^1 |h(\phi)|^p \ d\phi.
\end{align}
Similarly one can show 
\begin{align}
    \frac{1}{n} \sum_{m=1}^n \E{X_{m,k}}
\to
\int_0^1 h(\phi) \ d\phi.
\end{align}

Applying these arguments to the imaginary part of $w^{2,n}(f)$ completes the proof for absolutely continuous $z$.\\

{\it Discrete $z$.}

Recall that for discrete $z$ with $z^\rho\equiv1$, there exist always a sequence $(c_j)_{0\leq j\leq \rho-1}$ such that 
\begin{align}
  \Pb{z = e^{2 \pi i k/\rho}} = \frac{1}{\rho} \sum_{j=0}^{\rho-1} c_j e^{2 \pi i jk}.
\end{align}
(See for more details \cite{stein}, chapter 7.)
It follows immediately
\begin{align}
  \Pb{T_{m,1} = e^{2 \pi i k/\rho}} = \frac{1}{\rho} \sum_{j=0}^{\rho-1} c^m_j e^{2 \pi i jk}.
\end{align}
%
%
For any $p\geq 1$, we have
\begin{align}
 \frac{1}{n} \sum_{m=1}^n \E{\bigl|\log\bigl|f(x^m T_{m,1})  \bigr|\bigr|^p}
 &=
  \frac{1}{n} \sum_{m=1}^n \sum_{k=0}^{\rho-1} \bigl| \log\bigl|f(x^m e^{2 \pi i k/\rho})  \bigr|\bigr|^p \Pb{T_{m,1} =e^{2 \pi i jk/\rho}}
\nonumber\\
 &=
 \sum_{j=0}^{\rho-1}  \left(\frac{1}{n} \sum_{m=1}^n \left( \frac{1}{\rho} \sum_{k=0}^{\rho-1} c^m_j e^{2 \pi i jk}  \bigr|h(k/\rho+ m\varphi)|^p \right)\right)
\label{eq_discrete_W2_fourier}
\end{align}
%
Since $c_0=1$, we have that the summands corresponding to $j=0$ give
\begin{align}
\frac{1}{n} \sum_{m=1}^n \left( \frac{1}{\rho} \sum_{k=0}^{\rho-1} \bigr|h(k/\rho+ m\varphi)|^p \right)
=
\frac{1}{n} \sum_{m=1}^n  \bigr|h(k/\rho+ m\varphi)|^p.
\end{align}
We already mentioned in the proof of Theorem~\ref{thm_CLT_W2_1dim_general} that this expression converges to 
$\int_0^1  \bigr|h(\phi)|^p \ d\phi$. We now show that the remaining sum is $o(1)$.
Since by assumption $|c_j|<1$ for $j\neq 0$, we can find a $m_0$ such that $|c_j|^m < \eps$ for $m\geq m_0$ and all $j\neq 0$.  We thus get for all $m \geq m_0$,
\begin{align}
  \left|c^m_j \bigl|\log\bigl|f(x^m e^{2 \pi i k/\rho})\bigr| \bigr|^p \right|
&\leq
 \eps  \left| \log\bigl|f(x^m e^{2 \pi i k/\rho})\bigr|^p  \right|
\end{align}
Since $\epsilon$ was arbitrary and $\rho$ is finite, we see that the sum over all terms with $j\neq0$ is $o(1)$.
The remaining argumentation are the same as in the previous cases and we thus omit them.
%
%
%
\end{proof}

{\bf Proof of Corollary~\ref{cor_Zn_general_1dim}}
\begin{proof}
We use that $Z_{n,z}(x)= W^{2,n}_z(f)$ with $f(x)=1-x^{-1}$. Then Corollary~\ref{cor_Zn_general_1dim} is a direct application of Theorem~\ref{thm_CLT_W2_1dim_general}. 
One only has to compute $V_R(f)$, $V_I(f)$ and $m(f)$. 
A simple computation and Jensen's formula give in this case

\begin{align}
 V_R(f)
=
 V_I(f)
  =
  \frac{\pi^2}{12}
\text{ and }
m(f)
=
0. 
\end{align}
This completes the proof.

\end{proof}

\subsection{Behavior at different points}
\label{sec_CLT_in_ddim}

In this section, we study the joint behavior of the real and the imaginary parts of the characteristic polynomial of $M(\sigma, z)$ and of multiplicative class functions.
Furthermore, we consider the behavior at a finite set of different points $x_1=e^{2\pi i \varphi_1},\dots, x_d = e^{2\pi i \varphi_d}$, $d\in\N$ fixed. \\ 

Before we state the results of this section, it is important to emphasize that we will allow different random variables $z_1,\dots, z_d$ at the different points $x_1,\dots, x_d$. 
Of course, we need to specify the joint behavior at the different points. 
The idea is to define it in such a way that the behavior in disjoint cycles is still independent and the
behaviour in given cycle depends only on the cycle length.
For the multiplicative class function $w^{1,n}(f_j)(x_j)$, we define the following joint behavior. Let $\overline{z}= (z_1,\dots,z_d)$ be a random variable with values in $\mathbb{T}^d$.
Let further $\overline{z}^{(m,k)} = (z_{1}^{(m,k)},\dots,z_{d}^{(m,k)})$ be a sequence of i.i.d. random variables with $\overline{z}^{(m,k)} \stackrel{d}{=} \overline{z}$ (in $m$ and $k$, for $1\leq m\leq n$ and $1\leq k\leq C_m$, where $C_m$ denotes the number of cycles of $m$ in $\sigma$). Then, for functions $f_1,\dots, f_d$ and for any fixed $1\leq j\leq d$,

\begin{align}
w^{1,n} (f_j)(x_j)
=
w_{z_j}^{1,n} (f_j)(x_j)
:=
\sum_{m=1}^{n} \sum_{k=1}^{C_m}
\log\left(f_j\left(z_j^{(m,k)} x_j^m \right)\right).
\end{align}

As requested, we get with this definition that the behavior in disjoint cycles of $\sigma$ is independent. 
But the behavior in a given cycle at different points is determined by $\overline{z}$.\\

For the logarithm of the characteristic polynomial $\log\bigl(Z_{n,z}(x_j)\bigr)$ and for the multiplicative class function $w^{2,n}(f_j)(x_j)$, we do something similar. Intuitively, we construct for each point $x_j$ a matrix $M_{\sigma, z_j}$ as in \eqref{eq_def_wreath_product_matrix}, where we choose for $M_{\sigma, z_1}$ $n$ i.i.d. random variables, which are equal in distribution to $z_1$. At point $x_2$, we choose again $n$ i.i.d random variables, which are equal in distribution to $z_2$ and so on. 
Formally, we define for (the same sequence as above) $\overline{z}^{(m,k)} = (z_{1}^{(m,k)},\dots,z_{d}^{(m,k)})$ another sequence (in $m$ and in $k$) $\overline{T}^{(m,k)}=(T_{1}^{(m,k)},\dots,T_{d}^{(m,k)})$ of independent random variables, so that for any fixed $1\leq j\leq d$ and fixed $1\leq m\leq n$,

\begin{align}
(T_{1}^{(m,k)},\dots,T_{d}^{(m,k)})\stackrel{d}{=}\left(\prod_{\ell=1}^{m}z_1^{(m,\ell)},\dots,\prod_{\ell=1}^{m}z_d^{(m,\ell)}\right),
\end{align}

which implies
$$T_j^{(m,k)} \stackrel{d}{=} \prod_{\ell=1}^{m}z_j^{(m,\ell)}.$$ 

This gives for fixed $j$'s and function $f_j$:

\begin{align}
w^{2,n} (f_j)(x_j)
=
w_{z_j}^{2,n} (f_j)(x_j)
:=
\sum_{m=1}^{n} \sum_{k=1}^{C_m}
\log\left(f_j\left(T_j^{(m,k)} x_j^m \right)\right).
\end{align}

We now state the results of this section:
\begin{theorem}
\label{thm_CLT_W1_general_ddim}
  Let $S_n$ be endowed with the Ewens distribution with parameter $\theta$, $f_1,\dots,f_d$ be non zero real analytic functions, $\overline{z} = (z_1, \dots, z_d)$ a $\mathbb{T}^d$-valued random variable and $x_1=e^{2\pi i \varphi_1} ,\dots, x_d = e^{2\pi i \varphi_d} \in \mathbb{T}$ be such that $1, \varphi_1,\dots,\varphi_d$ are linearly independent over $\Z$.

Suppose that one of the following conditions is satisfied:

\begin{itemize}
  \item $z_1, \dots, z_d$ are uniformly distributed and independent.
  \item For all $1\leq j,\ell\leq d$ and $j\neq \ell$, the joint law of $(z_j,z_\ell)$ is absolutely continuous. The joint density of $z_j$ and $z_\ell$ is bounded and Riemann integrable for all $j\neq \ell$.
  \item For all $1\leq j\leq d$, $z_j$ is trivial, i.e. $z_j\equiv 1$, and all zeros of $f_j$ are roots of unity. Furthermore, $x_1,\dots,x_d$ are pairwise of finite type (see Definition~\ref{def_finite_type_pairwise}).
  \item For all $1\leq j\leq d$, there exists a $\rho_j>0$ with $\left(z_j\right)^{\rho_j} \equiv 1$, all zeros of $f_j$ are roots of unity and $x_1,\dots,x_d$ are pairwise of finite type.
\end{itemize}
We then have, as $n\to\infty$,
\begin{align*}
  \frac{1}{\sqrt{\log n}}
  \left(
\begin{array}{c}
 w^{1,n}(f_1)(x_1)\\
\vdots\\
w^{1,n}(f_d)(x_d)
\end{array}
\right)
-
\theta
\sqrt{\log n}
\left(
\begin{array}{c}
 m(f_1)\\
\vdots\\
m(f_d)
\end{array}
\right)
\stackrel{d}{\longrightarrow}
N
=
  \left(
\begin{array}{c}
 N_1\\
\vdots\\
N_d
\end{array}
\right),
\end{align*}
where $N$ is a $d-$variate complex normal distributed random variable with, for $j\neq\ell$,
\begin{align}
  \cov \bigl( \Re (N_j), \Re(N_\ell) \bigr)
  &=
  \theta\int_{[0,1]^2} \log\bigl|f_j(e^{2\pi i u}) \bigr| \log\bigl| f_\ell(e^{2\pi i v}) \bigr| \ du dv,
  \\
  \cov \bigl( \Re (N_j), \Im(N_\ell) \bigr)
  &=
  \theta\int_{[0,1]^2} \log\bigl|f_j(e^{2\pi i u}) \bigr| \arg\bigl( f_\ell(e^{2\pi i v}) \bigr) \ du dv,
  \\
  \cov \bigl( \Im (N_j), \Im(N_\ell) \bigr)
  &=
  \theta\int_{[0,1]^2} \arg\bigl( f_j(e^{2\pi i u}) \bigr)\arg\bigl( f_\ell(e^{2\pi i v}) \bigr) \ du dv.
\end{align}
and for $j=\ell$,
\begin{align}
\cov \bigl( \Re (N_j), \Re(N_j) \bigr)
&=
 \Var{\Re (N_j)}
  =
  \theta\int_{[0,1]} \log^2\bigl|f_j(e^{2\pi i u}) \bigr|\ du,\\
\cov \bigl( \Re (N_j), \Im(N_j) \bigr)
&=
 \Var{\Im(N_j)}
  =
  \theta\int_{[0,1]} \arg^2\bigl(f_j(e^{2\pi i v}) \bigr) \ dv.
\end{align}

\end{theorem}

Note that, in Theorem~\ref{thm_CLT_W1_general_ddim}, the first condition implies the second and the third condition implies the fourth.
For the multiplicative function $w^{2,n}$, we have the following result:
\begin{theorem}
\label{thm_CLT_W2_general_ddim}
  Let $S_n$ be endowed with the Ewens distribution with parameter $\theta$, $f_1,\dots, f_d$ be non zero real analytic functions, $\overline{z} = (z_1, \dots, z_d)$ a $\mathbb{T}^d$-valued random variable and $x_1=e^{2\pi i \varphi_1} ,\dots, x_d = e^{2\pi i \varphi_d} \in \mathbb{T}$ be such that $1, \varphi_1,\dots,\varphi_d$ are linearly independent over $\Z$.

Suppose that one of the following conditions is satisfied:

\begin{itemize}
  \item $z_1, \dots, z_d$ are uniformly distributed and independent.
  \item For all $1\leq j,\ell\leq d$ and $j\neq \ell$, the joint law of $(z_j,z_\ell)$ is absolutely continuous. For each $j \neq \ell$, the joint density $g_{j,\ell}$ of $z_j$ and $z_\ell$ satisfies 
  \begin{align}
g_{j,\ell}(\phi_j,\phi_\ell)
=
\sum_{a,b\in\Z} c_{a,b} e^{2 \pi i (a \phi_j + b \phi_\ell)}
\ \text{ and } \ \sum_{a,b\in\Z} |c_{a,b}| < \infty.
\end{align}
  \item For all $1\leq j\leq d$, $z_j$ is trivial, i.e. $z_j\equiv 1$, and all zeros of $f_j$ are roots of unity. Furthermore, $x_1,\dots,x_d$ are pairwise of finite type (see Definition~\ref{def_finite_type_pairwise}),
  \item For all $1\leq j\leq d$, $z_j$ is discrete, there exists a $\rho_j>0$ with $\left(z_j\right)^{\rho_j} \equiv 1$, all zeros of $f_j$ are roots of unity. Furthermore, assume that $x_1,\dots, x_d$ are pairwise of finite type (see Definition~\ref{def_finite_type_pairwise}) and that for $j\neq \ell$
\begin{align}
  \Pb{z_j = e^{2 \pi i k_1/\rho_j},z_\ell = e^{2 \pi i k_2/\rho_\ell} }
=
\frac{1}{\rho_j\rho_\ell} \sum_{a=0}^{\rho_j-1} \sum_{b=0}^{\rho_\ell-1}c_{a,b} e^{2 \pi i (ak_1+ bk_2)}\\
\ \text{ with }
\sum_{a=0}^{\rho_j-1} |c_{a,b}|<1, \text{ for } b\neq 0\ \text{ and }  \sum_{b=0}^{\rho_\ell-1} |c_{a,b}|<1, \text{ for } a\neq 0.
\end{align}
\end{itemize}
We then have, as $n\to\infty$,
\begin{align*}
  \frac{1}{\sqrt{\log n}}
  \left(
\begin{array}{c}
 w^{2,n}(f_1)(x_1)\\
\vdots\\
w^{2,n}(f_d)(x_d)
\end{array}
\right)
-
\theta
\sqrt{\log n}
\left(
\begin{array}{c}
 m(f_1)\\
\vdots\\
m(f_d)
\end{array}
\right)
\stackrel{d}{\longrightarrow}
N
=
  \left(
\begin{array}{c}
 N_1\\
\vdots\\
N_d
\end{array}
\right)
\end{align*}
with $m(f_j)$ and $N$ as in Theorem~\ref{thm_CLT_W1_general_ddim}.
\end{theorem}
As before, we get as simple corollary, which covers Proposition~\ref{prop_CLT_Zn_general_ddim}:
\begin{corollary}
\label{thm_CLT_Zn_general_ddim}
  Let $S_n$ be endowed with the Ewens distribution with parameter $\theta$, $\overline{z} = (z_1, \dots, z_d)$ be a $\mathbb{T}^d$-valued random variable and $x_1=e^{2\pi i \varphi_1} ,\dots, x_d = e^{2\pi i \varphi_d} \in \mathbb{T}$ be such that $1, \varphi_1,\dots,\varphi_d$ are linearly independent over $\Z$.

Suppose that one of the conditions in Theorem~\ref{thm_CLT_W2_general_ddim} is satisfied:
We then have, as $n\to\infty$,
\begin{align*}
\frac{1}{\sqrt{\frac{\pi^2}{12}\theta\log n}}
  \left(
\begin{array}{c}
 \log(Z_{n, z_1}(x_1)\bigr)\\
\vdots\\
\log \bigl(Z_{n, z_d}(x_d) \bigr)
\end{array}
\right)
\stackrel{d}{\longrightarrow}
  \left(
\begin{array}{c}
N_1\\
\vdots\\
N_d
\end{array}
\right)
\end{align*}
with $\Re(N_1),\dots,\Re(N_d), \Im(N_1),\dots,\Im(N_d)$ independent standard normal distributed random variables.
\end{corollary}
{\bf Proof of Theorem~\ref{thm_CLT_W1_general_ddim} and~\ref{thm_CLT_W2_general_ddim}} 
\begin{proof}
We consider $w^{1,n}(f)$ and $w^{2,n}(f)$ as $\R^2$-valued random variables and argue with Theorem~\ref{thm_general_CLT_ddim}.
Using the verified conditions \eqref{eq:V_N} and \eqref{eq:inf_var_cond} from Theorem~\ref{thm_trace_infinite_case}, all conditions of Theorem~\ref{thm_general_CLT_ddim} are satisfied if the following equation is true:
%
\begin{align}
  \lim_{n\to\infty}\frac{1}{n}\sum_{m=1}^n \E{X_{m,1,j}X_{m,1,\ell}} = \sigma_{j,\ell}.
  \label{eq_cond_ii_general_2}
\end{align}
The computations for uniformly distributed and for absolute continuous $z_1,\dots, z_d$ are for both, $w^{1,n}$ and $w^{2,n}$, the same as 
in the proof of Theorem~\ref{thm_CLT_W1_1dim_general} and the proof of Theorem~\ref{thm_CLT_W2_1dim_general} and we thus omit them.
The trivial and the discrete case (the third and the forth condition in in Theorem~\ref{thm_CLT_W1_general_ddim}) is slightly more difficult 
and we thus have a closer look at them.
The behavior in one point, where $z\equiv 1$ has been treated by \cite{HKOS}. For the behavior at different points, we need the following lemma:

\begin{lemma}
\label{lem_conv_log_finite_type_d=2}
Let $f_1, f_2:\mathbb{T} \to \C$ be real analytic with only roots of unity as zeros and let $x_1= e^{2 \pi i \varphi_1}$ and $x_2= e^{2 \pi i \varphi_2}$ be such that $(x_1, x_2)\in\mathbb{T}^2$ be of finite type (see Definition~\ref{def_finite_type}). We then have $\log\bigl| f_j(x_j^n)\bigr| = O\bigl(\log n\bigr)$ for $j\in\{1,2\}$. Moreover, as $n\to\infty$
\begin{align}
\frac{1}{n} \sum_{m=1}^n \log\left| f_1\left( x_1^m\right)\right|    \log\left|f_2\left( x_2^m \right) \right|
 &\longrightarrow
\int_{[0,1]^2} \log\bigl|f_1(e^{2\pi i u}) \bigr| \log\bigl| f_2(e^{2\pi i v}) \bigr| \ du dv,
\label{eq_convergence_log_finite_type_d=2}
\end{align}

\begin{align}
\frac{1}{n} \sum_{m=1}^n \arg\left( f_1\left( x_1^m\right)\right)    \log\left|f_2\left( x_2^m \right) \right|
 &\longrightarrow
\int_{[0,1]^2} \arg\bigl(f_1(e^{2\pi i u}) \bigr) \log\bigl| f_2(e^{2\pi i v}) \bigr| \ du dv
\label{eq_convergence_arg_finite_type_d=2}
\end{align}
and 
\begin{align}
\frac{1}{n} \sum_{m=1}^n \arg\left( f_1\left( x_1^m\right)\right)    \arg\left(f_2\left( x_2^m \right) \right)
 &\longrightarrow
\int_{[0,1]^2} \arg\bigl(f_1(e^{2\pi i u}) \bigr) \arg\bigl( f_2(e^{2\pi i v}) \bigr) \ du dv.
\label{eq_convergence_argarg_finite_type_d=2}
\end{align}
\end{lemma}

By using Lemma~\ref{lem_conv_log_finite_type_d=2}, the proof for $z_1, \dots, z_d$ being discrete is the same as the discrete case in the proof of Theorem~\ref{thm_CLT_W1_1dim_general} and the proof of Theorem~\ref{thm_CLT_W2_1dim_general}. Thus, in order to conclude the proofs for Theorem~\ref{thm_CLT_W1_general_ddim} and~\ref{thm_CLT_W2_general_ddim}, we will proceed by giving the proof of Lemma~\ref{lem_conv_log_finite_type_d=2}:

\begin{proof}
We start by considering \eqref{eq_convergence_log_finite_type_d=2}. 
Since $x_1$ and $x_2$ are not roots of unity, we expect for $n \to \infty$
\begin{align}
 \frac{1}{n} \sum_{m=1}^n \log\left| f_1\left( x_1^m\right)\right|    \log\left|f_2\left( x_2^m \right) \right|
 \longrightarrow
 \int_{[0,1]^2} \log\bigl|f_1(e^{2\pi i u}) \bigr| \log\bigl| f_2(e^{2\pi i v}) \bigr| \ du dv.
 \label{eq_conv_covar_trivial}
 \end{align}
Unfortunately this is not automatically true since $\log(f_j)$ is not of bounded variation if $f_j$ has zeros and we thus cannot apply Theorem~\ref{thm_equidist_integral_convergence}.
We show here that \eqref{eq_conv_covar_trivial} is true by using Theorem~\ref{thm_koksma_inequality_2} and the assumption that $(x_1, x_2)$ is of finite type. 

We use the notations:
\begin{align}
h_1(\phi)&:= \log|f_1(e^{2\pi i \phi })|,\
\varphi_1^{(m)} := \fracpart{ m \varphi_1},  \
\vB_1 := ( \varphi_1^{(m)})_{m=1}^\infty,
\nonumber\\
h_2(\phi)&:= \log|f_2(e^{2\pi i \phi })|,\
\varphi_2^{(m)} := \fracpart{ m \varphi_2},  \
\vB_2 := ( \varphi_2^{(m)})_{m=1}^\infty,
\nonumber\\
\overline{\varphi}&:= (\varphi_1,\varphi_2), \
\varphi^{(m)} := (\varphi_1^{(m)},\varphi_2^{(m)}) \
\vB:= \bigl( \varphi^{(m)}\bigr)_{m=1}^\infty.
\end{align}
We thus can reformulate the LHS of \eqref{eq_convergence_log_finite_type_d=2} as
\begin{align}
 \frac{1}{n} \sum_{m=1}^n h_1(\varphi_1^{(m)}) h_2(\varphi_2^{(m)}).
\end{align}

If $f_1$ and $f_2$ are zero free, then $h_1$ and $h_2$ are Riemann integrable and of bounded variation. Furthermore, $1, \varphi_1, \varphi_2$ are by assumption linearly independent over $\Z$, and thus $\vB$ is a uniformly distributed sequence by Lemma~\ref{lem_all_irrational_uniform_distributed}. Equation \eqref{eq_conv_covar_trivial} now follows immediately with Theorem~\ref{thm_equidist_integral_convergence}.\\


If $f_1$ and $f_2$ are not zero free, we have to be more careful.
We use in this case Theorem~\ref{thm_koksma_inequality_2} for $d=2$. We assume for simplicity that $0$ and $1$ are to the only singularities of $h_1$ and $h_2$. The more general case with roots of unity as zeros is completely similar.

We first have to choose a suitable $\delta = \delta(n)$ such that $\varphi^{(m)} \in [\delta, 1- \delta]^{2}$ for $1 \leq m \leq n$.
Since by assumption $\overline{\varphi}$ is of finite type, there exists $K>0,\gamma>1$ such that
 \begin{align}
    \norm{\overline{q}\cdot \overline{\varphi}} \geq \frac{K}{(\norminf{\overline{q}})^\gamma}
    \ \text{ for all }\overline{q}\in \Z^2\setminus\set{0}
  \end{align}
with $\norm{a}:= \inf_{m\in\Z} |a-m|$. We thus can chose $\delta = \frac{K}{n^\gamma}$.

Next, we have to estimate the discrepancies of the sequences $\vB_1, \vB_2$ and $\vB$.
Since $\overline{\varphi},\varphi_1,\varphi_2$ are of finite type, we can use Theorem~\ref{thm_estimate_Dn_finite_type} and get
\begin{align}
  D_n^{*} (\vB_1) = O(n^{-\alpha_1}), \
  D_n^{*} (\vB_2) = O(n^{-\alpha_2}) \ \text{ and } \
  D_n^{*} (\vB) = O(n^{-\alpha}) 
\end{align}
for some $\alpha_1,\alpha_2,\alpha>0$.

We can show now with Theorem~\ref{thm_koksma_inequality_2} that the error made by the approximation in \eqref{eq_conv_covar_trivial} goes to $0$ by showing that all summands on the RHS of \eqref{eq_koksma_inequality_2} go to $0$.
This computation is straightforward and very similar for each summand.
We restrict ourselves to illustrate the computations only on the summands corresponding to the face $F$ of $[\delta,1-\delta]^2$ with $\phi_1 = 1-\delta$. We get with $h(\phi_1,\phi_2):= h_1(\phi_1) h_2(\phi_2)$,
\begin{align}
  \delta |h_1(1-\delta)| \int_{\delta}^{1-\delta} |h_2(u)| \ du + D_n^*(\vB_2) |h_1(1-\delta)| V(h_2|{[\delta,1-\delta]}),
  \label{eq_compute_on_one_face}
\end{align}
where $V(h_2|{[\delta,1-\delta]})$ is the variation of $h_2|{[\delta,1-\delta]}$.
It is easy to see that, for $\phi \to 0$ and some $K_1>0$, $h_1(\phi) \sim K_1 \log(\phi) \sim h_1(1-\phi)$ .
Thus, the first summand in \eqref{eq_compute_on_one_face} goes to $0$ for $n \to \infty$.
On the other hand we have
\begin{align}
  &D_n^*(\vB_2) |h_1(1-\delta)| V(h_2|{[\delta,1-\delta]}) \sim K_2 D_n^*(\vB_2) \log^2 \delta \leq K_3 n^{-\alpha_2} \log^2 n.
\end{align}
for constants $K_2, K_3 >0$. This shows that also the second term in \eqref{eq_compute_on_one_face} goes to $0$. So, we proved \eqref{eq_convergence_log_finite_type_d=2}. Equations \eqref{eq_convergence_arg_finite_type_d=2} and \eqref{eq_convergence_argarg_finite_type_d=2} are straightforward, with the given computations above and we conclude Lemma~\ref{lem_conv_log_finite_type_d=2}.
\end{proof}
This completes the proofs of Theorem~\ref{thm_CLT_W1_general_ddim} and~\ref{thm_CLT_W2_general_ddim}.
\end{proof}

\bibliographystyle{plain}

\end{document}